\documentclass[reqno]{amsart}

\usepackage{graphics}
\usepackage{latexsym}
\usepackage{amsfonts}
\usepackage{amsmath}
\usepackage{amssymb}
\usepackage{amsthm}
\usepackage{multicol}
\usepackage{dsfont}
\usepackage{enumerate}

\numberwithin{equation}{section}

\newtheorem{theorem}{Theorem}[section] % Sets up theorem environment
\newtheorem{proposition}{Proposition}[section] % Sets up theorem environment
\newtheorem{lemma}{Lemma}[section] % Sets up lemma environment
\newtheorem{remark}{Remark}[section] % Sets up remark environment

\newcommand{\indicator}[1]{\ensuremath{1_{\{#1\}}}}

\newcommand{\R}{\mathbb R} %REALS
\newcommand{\C}{\mathbb C} %COMPLEX
\newcommand{\T}{\mathbb T} %TORUS

\renewcommand{\P}{\mathbb P}
\newcommand{\E}{\mathbb E}
\newcommand{\V}{\mathbb V}
\newcommand{\Cov}{\mathbb Cov}

\newcommand{\tr}{\text{tr}}
\newcommand{\Tr}{\text{Tr}}

\renewcommand\Re{\operatorname{\mathfrak{Re}}}
\renewcommand\Im{\operatorname{\mathfrak{Im}}}

\begin{document}
\title[On Fluctuations of Matrix Entries]{On Fluctuations of Matrix Entries of Regular Functions of Wigner Matrices with 
Non-Identically Distributed Entries}
\author[S. O'Rourke]{Sean O'Rourke}
\address{Department of Mathematics, University of California, Davis, One Shields Avenue, Davis, CA 95616-8633  }
\thanks{S.O'R. has been supported in part by the NSF grants  VIGRE DMS-0636297 and DMS-1007558}
\email{sdorourk@math.ucdavis.edu}

\author[D. Renfrew]{David Renfrew} \thanks{D.R. has been supported in part by the NSF grants VIGRE DMS-0636297, DMS-1007558, and DMS-0905988 }
\address{Department of Mathematics, University of California, Davis, One Shields Avenue, Davis, CA 95616-8633  }
\email{drenfrew@math.ucdavis.edu}

\author[A. Soshnikov]{Alexander Soshnikov}
\address{Department of Mathematics, University of California, Davis, One Shields Avenue, Davis, CA 95616-8633  }
\thanks{A.S. has been supported in part by the NSF grant DMS-1007558}
\email{soshniko@math.ucdavis.edu}

\begin{abstract}
In this note, we extend the results about the fluctuations of the matrix entries of regular functions of Wigner random matrices 
obtained in \cite{PRS} to Wigner matrices with non-i.i.d. entries provided certain Lindeberg type conditions for the fourth moments are satisfied.  
In addition, we relax our conditions on the test functions and require that for some $s>3$ \\
$ \int_{\R} (1+ 2|k|)^{2s}\*|\hat{f}(k)|^2 \* dk <\infty.$
\end{abstract}

\maketitle

\section{ \bf{Introduction and Formulation of Main Results}}
\label{sec:intro}

Let $X_N= \frac{1}{\sqrt{N}} W_N $ be a random Wigner real symmetric (Hermitian) matrix.
In the real symmetric case, we assume that the entries 
$$(W_N)_{jk},\ 1\leq j\leq k \leq N,$$ 
are independent random variables
such that the off-diagonal entries satisfy
\begin{equation}
\label{offdiagreal}
\E (W_N)_{jk}=0, \ \V(W_N)_{jk}=\sigma^2, \ 1\leq j<k \leq N, \ m_4:=\sup_{j\not=k,N} \E (W_N)_{jk}^4<\infty,
\end{equation}
and the Lindeberg type condition for the fourth moments takes place,
\begin{equation}
\label{lind1}
L_N(\epsilon) \to 0, \ \text{as} \ N\to \infty, \ \forall \epsilon>0, 
\end{equation}
where 
\begin{equation}
\label{lind2}
L_N(\epsilon)= \frac{1}{N^2}\* \sum_{1\leq i<j\leq N} \E\left( |(W_N)_{ij}|^4 \*\indicator{|(W_N)_{ij}|\geq \epsilon\*\sqrt{N}}\right ).
\end{equation}
Here and throughout the paper, $\E \xi $ denotes the mathematical expectation and $\V \xi $ the variance of a random variable $\xi.$

In addition, we assume that the diagonal entries  
satisfy
\begin{align}
\label{diagreal}
& \E (W_N)_{ii}=0, \ \ 1\leq i \leq N, \ \sigma^2_1:=\sup_{i,N} \E |(W_N)_{ii}|^2 < \infty,\\
\label{diagreal1}
& l_N(\epsilon)\to 0, \ \text{as} \ N\to \infty, \ \forall \epsilon>0, \ \text{where} \\
\label{diagreal2}
& l_N(\epsilon)= \frac{1}{N}\* \sum_{1\leq i \leq N} \E\left( |(W_N)_{ii}|^2 \*\indicator{|(W_N)_{ii}|\geq \epsilon\*\sqrt{N}}\right ).
\end{align}

We note that (\ref{lind1}) and (\ref{diagreal1})  are satisfied if
\begin{equation}
\label{foureps}
\sup_{i\not=j,N} \E |(W_N)_{ij}|^{4+\epsilon}<\infty, \ \sup_{i,N} \E |(W_N)_{ii}|^{2+\epsilon}<\infty.
\end{equation}

If $\{\frac{1}{\sqrt{2}}\*(W_N)_{ii}, 1\leq i \leq N, \ (W_N)_{jk}, \ 1\leq j < k \leq N, \}$ are i.i.d. $N(0, \sigma^2)$ random variables, $X_N$ 
belongs to the Gaussian Orthogonal Ensemble (GOE).

In the Hermitian case, we assume that the entries  
$$\Re (W_N)_{jk}, \ \Im (W_N)_{jk}, \ 1\leq j<k \leq N, \ (W_N)_{ii}, \ 1\leq i \leq N,$$
are independent random variables such that
the off-diagonal entries satisfy
\begin{align}
\label{offdiagherm1}
& \E \Re (W_N)_{jk}= \E \Im (W_N)_{jk} = 0, \ \ 1\leq j<k \leq N, \\
\label{offdiagherm2}
& \V  \Re (W_N)_{jk} = \V \Im (W_N)_{jk}= \frac{\sigma^2}{2}, \ 1\leq j<k \leq N, \ m_4:=\sup_{j\not=k, N} \E |(W_N)_{jk}|^4<\infty,
\end{align}
and the Lindeberg type condition (\ref{lind1}) for the fourth moments of the off-diagonal entries takes place.

In addition, we assume that the diagonal entries  satisfy
\begin{equation}
\label{diagherm}
\E (W_N)_{ii}=0, \ \ 1\leq i \leq N, \ \sigma^2_1:=\sup_{i,N} \E |(W_N)_{ii}|^2<\infty,
\end{equation}
and the Lindeberg type condition (\ref{diagreal1}) for the second moments of the diagonal entries takes place.

If $\{\frac{1}{\sqrt{2}}\*(W_N)_{ii}, 1\leq i \leq N, \ \Re (W_N)_{jk}, \ \Im (W_N)_{jk}, 
\ 1\leq j < k \leq N, \}$ are i.i.d. $N(0, \frac{\sigma^2}{2})$ random variables, $X_N$ 
belongs to the Gaussian Unitary Ensemble (GUE).

We define the empirical distribution of the eigenvalues of $X_N$ as
\begin{equation}
\label{edf}
\mu_{X_N} = \frac{1}{N} \sum_{i=1}^{N} \delta_{\lambda_{i}},
\end{equation}
where  $\lambda_1 \leq \ldots \leq \lambda_N$ are the (ordered) eigenvalues of $X_N.$

Wigner semicircle law (see e.g. \cite{wig}, \cite{BG}, \cite{AGZ}, \cite{B}) 
states that the random measure $\mu_{X_N}(dx,\omega)$ converges almost surely in distribution
to the (non-random) Wigner semicircle distribution  $\mu_{sc}.$  The limiting distribution is absolutely continuous with respect to the Lebesgue 
measure and its density is given by
\begin{equation}
\label{polukrug}
\frac{d \mu_{sc}}{dx}(x) = \frac{1}{2 \pi \sigma^2} \sqrt{ 4 \sigma^2 - x^2} \mathbf{1}_{[-2 \sigma , 2 \sigma]}(x).
\end{equation}
Its Stieltjes transform
\begin{equation}
\label{steltsem} 
g_\sigma(z) := \int \frac{d \mu_{sc}(x)}{z-x}= \frac{z-\sqrt{z^2-4\*\sigma^2}}{2\*\sigma^2}, \ z \in \C \backslash [-2\*\sigma, 2\*\sigma].
\end{equation}
is the solution to 
\begin{equation}
\label{semicircle}
\sigma^2 g_\sigma^2(z) - z g_\sigma(z) + 1 = 0
\end{equation}
that decays to 0 as $|z| \to \infty$.

This paper is devoted to the question of the fluctuations of  matrix entries of $f(X_N)$ for regular test functions $f.$
Lytova and Pastur (\cite{LP}) considered the GOE/GUE case and proved that
\begin{equation}
\label{roscha}
\sqrt{N} \* \left( f(X_N)_{ij} - \E (f(X_N)_{ij})\right) \to N(0, \frac{1+\delta_{ij}}{\beta}\* \omega^2(f)),
\end{equation}
with $\beta=1(2) $ in the GOE (GUE) case,
\begin{equation}
\label{omegasq}
\omega^2(f):= \V(f(\eta))= \frac{1}{2}\* \int_{-2\sigma}^{2\sigma} \int_{-2\sigma}^{2\sigma} (f(x)-f(y))^2 
\*  \frac{1}{4 \pi^2 \sigma^4} \sqrt{ 4 \sigma^2 - x^2} \* 
\sqrt{ 4 \sigma^2 - y^2} \* dx \* dy,
\end{equation}
where $\eta$ is distributed according to the Wigner semicircle law (\ref{polukrug}).

In \cite{PRS}, Pizzo, Renfrew, and Soshnikov considered the non-Gaussian case and proved the following theorems.

\begin{theorem}[Theorem 1.3 in \cite{PRS}]
\label{thm:real}
Let $X_N=\frac{1}{\sqrt{N}} W_N$ be a random real symmetric Wigner matrix
(\ref{offdiagreal}), (\ref{diagreal}) 
such that the off-diagonal entries $(W_N)_{jk}, 1\leq j<k\leq N, $ are i.i.d. random variables with probability distribution $\mu$ and
the diagonal entries $(W_N)_{ii}, \ 1\leq i \leq N, $ are i.i.d. random variables with probability distribution $\mu_1.$ 

Let $f:\R \to  \R $ be four times continuously differentiable on $[-2\*\sigma -\delta, 2\*\sigma +\delta] $ for some $\delta>0\ $  and
$h(x)$ be a $C^{\infty}(\R) $ function with compact support such that
\begin{equation}
\label{amerika}
h(x)\equiv 1 \ \text{for} \  x \in [-2\*\sigma -\delta, 2\*\sigma +\delta], \ \delta>0.
\end{equation}
Then the following holds.

(i) For $ i =j, $
\begin{equation}
\label{trudno1}
\sqrt{N} \*\left( f(X_N)_{ii} - \E\left( (fh)(X_N)_{ii}\right)   \right) - \frac{\alpha(f)}{\sigma}\* (W_N)_{ii} \to N(0, v_1^2(f)),
\end{equation}
in distribution as $N \to \infty.$
where 
\begin{align}
\label{vsqf}
& v_1^2(f):= 2 \* \left(\omega^2(f) -\alpha^2(f) +\frac{\kappa_4(\mu)}{2\*\sigma^4} \* \beta^2(f)\right), \\
\label{alphaf}
& \alpha(f):=\E\left(f(\eta)\* \frac{\eta}{\sigma}\right)= 
\frac{1}{\sigma}\* \int_{-2\sigma}^{2\sigma} x\*f(x)\* \frac{1}{2 \pi \sigma^2} \sqrt{ 4 \sigma^2 - x^2}  \* dx, \\
\label{betaf}
& \beta(f):= \E\left(f(\eta)\*\frac{\eta^2-\sigma^2}{\sigma^2}\right)= \frac{1}{\sigma^2}\int_{-2\sigma}^{2\sigma} f(x)\*(x^2-\sigma^2)\* 
\frac{1}{2 \pi \sigma^2} \sqrt{ 4 \sigma^2 - x^2},
\end{align}
$\omega^2(f)$ defined in  (\ref{omegasq}), and 
$\kappa_4(\mu)$ is the fourth cumulant of $\mu,$
\begin{equation*}
\kappa_4(\mu)= \int x^4 \*\mu(dx) - 3\* (\int x^2 \* \mu(dx))^2=\E |(W_N)_{12}|^4-3\sigma^4.
\end{equation*}

If $f$ is seven times continuously differentiable on $[-2\sigma-\delta, 2\sigma+\delta], $ then one can replace
$\E\left( (fh)(X_N)_{ii}\right)$ in (\ref{trudno1}) by
\begin{equation}
\label{c1f}
\int_{-2\sigma}^{2\sigma} f(x) \* \frac{1}{2 \pi \sigma^2} \sqrt{ 4 \sigma^2 - x^2}  \* dx.
\end{equation}

(ii) For $ i \not=j, $
\begin{equation}
\label{trudno2}
\sqrt{N} \* \left( f(X_N)_{ij} - \E\left( (fh)(X_N)_{ij}\right)\right)- \frac{\alpha(f)}{\sigma} \*(W_N)_{ij} \to N(0, d^2(f))
\end{equation}
in distribution as $N \to \infty, $
where
\begin{equation}
\label{dsqf}
d^2(f):= \omega^2(f) -\alpha^2(f).
\end{equation}
If $f$ is six times continuously differentiable on $[-2\sigma-\delta, 2\sigma+\delta], $ then one can replace
$\E\left( (fh)(X_N)_{ij}\right)$ in (\ref{trudno2}) by $0.$

(iii) For any finite $m, $ the normalized matrix entries 
\begin{equation}
\label{mnogo}
\sqrt{N} \*\left( f(X_N)_{ij} - \E ((fh)(X_N)_{ij})\right), \ 1\leq i \leq j \leq m, 
\end{equation}
are independent in the limit $N \to \infty. $ 
\end{theorem}
\begin{remark}
If $f \in C^4(\R)$ and $\|f\|_{4,1} <\infty, $ where 
\begin{equation}
\label{normsob}
\|f\|_{n,1}:= \max_{0\leq k \leq n} \left( \int_{-\infty}^{\infty} |d^kf/dx^k(x)| \* dx \right)<\infty, 
\end{equation}
then one can replace $\E\left( (fh)(X_N)_{ij}\right)$ in (\ref{trudno1}-\ref{trudno2}) by
$\E (f(X_N))_{ij}. $
\end{remark}
In the Hermitian case, the analogue of Theorem \ref{thm:real} was proved in Theorem 1.7 of \cite{PRS}.

\begin{theorem}
\label{thm:herm}[Theorem 1.7 in \cite{PRS}]
Let $X_N=\frac{1}{\sqrt{N}} W_N$ be a random Hermitian Wigner matrix (\ref{offdiagherm1}-\ref{diagherm}),
such that the off-diagonal entries $(W_N)_{jk}, 1\leq j<k\leq N, $ are i.i.d. complex random variables with probability distribution $\mu$
and the diagonal entries $(W_N)_{ii}, \ 1\leq i \leq N, $ are i.i.d. random variables with probability distribution $\mu_1.$ 

Let $f:\R \to  \R $ be four times continuously differentiable on $[-2\*\sigma -\delta, 2\*\sigma +\delta] $ for some $\delta>0,\ $ 
and $h(x)$ be a $C^{\infty}(\R) $ function with compact support satisfying (\ref{amerika}).
Then the following holds.

(i) For $ i =j, $
\begin{equation}
\label{trudno11}
\sqrt{N} \*\left( f(X_N)_{ii} - \E\left( (fh)(X_N)_{ii}\right)\right) - \frac{\alpha(f)}{\sigma}\* (W_N)_{ii} \to N(0, v_2^2(f))
\end{equation}
in distribution as $N \to \infty,$
where

\begin{equation}
\label{vsqfherm}
v_2^2(f):= \omega^2(f) -\alpha^2(f) +\frac{\kappa_4(\mu)}{\sigma^4} \* \beta^2(f), 
\end{equation}
$ \omega^2(f), \alpha(f),$ and $ \beta(f)$ are defined in (\ref{omegasq}), (\ref{alphaf}), and (\ref{betaf}),
and $\kappa_4(\mu)$ is given by
\begin{equation*}
\kappa_4(\mu):= \E |(W_N)_{12}|^4-2\sigma^4.
\end{equation*}

If $f$ is seven times continuously differentiable on $[-2\sigma-\delta, 2\sigma+\delta], $ then one can replace $\E\left( (fh)(X_N)_{ii}\right) $
in (\ref{trudno11}) by (\ref{c1f}).

(ii) For $ i \not=j, $
\begin{equation}
\label{trudno12}
\sqrt{N} \* \left( f(X_N)_{ij} - \E\left( (fh)(X_N)_{ij}\right)\right) - \frac{\alpha(f)}{\sigma} \*(W_N)_{ij}\to N(0, d^2(f)),
\end{equation}
in distribution as $N \to \infty, $ where $N(0, d^2(f))$ stands for the complex Gaussian random variable with 
with i.i.d real and imaginary parts $N(0, \frac{1}{2}\*d^2(f)),$ and 
$d^2(f)$ defined in (\ref{dsqf}).

If $f$ is six times continuously differentiable on $[-2\sigma-\delta, 2\sigma+\delta], $ then one can replace $\E\left( (fh)(X_N)_{ij}\right)$
in (\ref{trudno12}) by $0.$

(iii) For any finite $m, $ the normalized matrix entries 
\begin{equation}
\label{mnogoh}
\sqrt{N} \*\left( f(X_N)_{ij} - \E ((fh)(X_N)_{ij})\right), \ 1\leq i \leq j \leq m, 
\end{equation}
are independent in the limit $N \to \infty. $ 
\end{theorem}

Almost simultaneously with \cite{PRS}, Pastur and Lytova (see Theorem 3.4 in \cite{LytP}) 
extended the technique of \cite{LP} and proved the convergence in distribution
for the normalized diagonal entries $\sqrt{N} \* \left( f(X_N)_{ii} - \E\left( f(X_N)_{ii}\right)\right), \ 1\leq i \leq N, $ when
the real symmetric Wigner matrix $X_N$ has i.i.d. entries up from the diagonal and, in addition to the requirements of Theorem \ref{thm:real}, 
the cumulant generating functions $\log\left(\E e^{zW_{12}}\right)$ is entire.  The results of \cite{LytP} hold provided the test function satisfies
\begin{equation*}
\int_{\R} (1+ 2\*|k|)^3\*|\hat{f}(k)| \* dk <\infty, 
\end{equation*}
where $\hat{f}(k)$ is the Fourier transform
\begin{equation}
\label{fourier}
\hat{f}(k)=\frac{1}{\sqrt{2\pi}}\*\int_{\R} e^{-i\*k\*x}\* f(x)\*dx.
\end{equation}
The approaches of \cite{PRS} and \cite{LytP} are independent from each other.  In particular, Pastur and Lytova prove the convergence of the 
characteristic function of $$\sqrt{N} \* \left( f(X_N)_{ii} - \E\left( f(X_N)_{ii}\right)\right).$$  

In addition, in the non-i.i.d. case, Theorem 3.2 of \cite{LytP} proves that
\begin{equation*}
\V [\sqrt{N} \*\left( f(X_N)_{ii} - \E ( f(X_N)_{ii}) \right)] \to 2\*v^2(f)
\end{equation*}
provided the matrix entries $(W_N)_{ij}$ are independent up from the diagonal
and satisfy
\begin{align}
\label{kharkiv}
& \E (W_N)_{jk}=0, \ \V(W_N)_{jk}=\sigma^2, \ \E (W_N)_{jk}^3=m_3, \E (W_N)_{jk}^4=m_4<\infty, \\
\label{kharkiv1}
& \sup_{j,k,N} \E |(W_N)_{jk}|^6 <\infty.
\end{align}

In this paper, we extend Theorems \ref{thm:real}  and \ref{thm:herm} to the non-i.i.d. setting provided the matrix entries
satisfy the fourth moment Lindeberg type conditions (\ref{lind1}) and (\ref{LIND1}) for the off-diagonal entries and the second moment 
Lindeberg type condition (\ref{diagreal1}) for the diagonal entries.  
Moreover, we relax the smoothness condition imposed in \cite{PRS} on the test function.

Consider the space $\mathcal{H}_s$
consisting of the functions $\phi:\R \to \R$  that satisfy
\begin{equation}
\label{sobolev}
\|\phi\|^2_s:= \int_{\R} (1+ 2|k|)^{2s}\*|\hat{\phi}(k)|^2 \* dk <\infty.
\end{equation}

The result below is valid (see Remark \ref{zamechanie}) provided a test function $f$ coincides on the interval 
$[-2\sigma-\delta, 2\sigma+\delta]$ with some function from 
 $\mathcal{H}_s$ for some $s>3, \ \delta>0.$  Thus, roughly speaking, we require that $f$ has
$3+\epsilon$ derivatives  on $[-2\sigma-\delta, 2\sigma+\delta].$

We recall that $C^n(\R)$ and $C^n([-L, L]) $ denote the spaces of $n$ times continuously differentiable functions on $\R$ and $[-L, L],$ respectively.
We define the norm on $C^n([-L, L]) $ as
\begin{equation}
\label{cnnorm}
\|f\|_{C^n([-L,L])}:= \max\left(|\frac{d^lf}{dx^l}(x)|, \ x \in [-L, L], \  0\leq l \leq n \right).
\end{equation}

\begin{theorem}
\label{thm:main}
Let $X_N=\frac{1}{\sqrt{N}} W_N$ be a random real symmetric (Hermitian) Wigner matrix
(\ref{offdiagreal}), (\ref{diagreal}) (respectively (\ref{offdiagherm1}-\ref{diagherm})  such that
the Lindeberg type condition (\ref{lind1}) for the fourth moments of the off-diagonal entries and 
the Lindeberg type condition (\ref{diagreal1}) for the second moments of the diagonal entries
are satisfied.
Let $f\in \mathcal{H}_s,$  for some $s>3.$
Let $m$ be a fixed positive integer, and for $1\leq i\leq m,$ assume that the following two conditions hold:

$(A_1)$
\begin{equation}
\label{LIND1}
\mathcal{L}_{i,N}(\epsilon) \to 0, \ \text{as} \ N\to \infty, \ \forall \epsilon>0, 
\end{equation}
where
\begin{equation}
\label{LIND2}
\mathcal{L}_{i,N}(\epsilon)= \frac{1}{N}\* \sum_{j: j\not=i} 
\E\left( |(W_N)_{ij}|^4 \*\indicator{|(W_N)_{ij}|\geq \epsilon\*N^{1/4}}\right );
\end{equation}
$(A_2)$
\begin{equation}
\label{AI}
m_4(i):=\lim_{N\to \infty} \frac{1}{N} \*\sum_{j:j\not=i} \E |(W_N)_{ij}|^4
\end{equation}
exists.

Then the results (i)-(iii) of Theorem \ref{thm:real} (respectively Theorem \ref{thm:herm}) hold for the joint distribution
of the matrix entries  $ \{\sqrt{N} \* \left( f(X_N)_{ij} - \E\left( f(X_N)_{ij}\right)\right), \ 1\leq i\leq j\leq m \},$
where $\kappa_4(\mu)$ must be replaced in (\ref{vsqf}) by 
\begin{equation}
\label{KIreal}
\kappa_4(i):=m_4(i)-3\*\sigma^4, \ 1\leq i\leq m,
\end{equation}
in the real symmetric case and by
\begin{equation}
\label{KIherm}
\kappa_4(i):=m_4(i)-2\*\sigma^4, \ 1\leq i\leq m,
\end{equation}
in the Hermitian case.

In addition, the following estimates for $\E (f(X_N)_{ij})$ take place.

(iv) Let $f:\R \to  \R $ belong to $C_c^7(\R),$ the space of  seven times continuously differentiable functions with compact support,
and $supp(f)\in [-L, L]$ for some $L>0.$ Then there exists a constant $Const_1(L, \sigma, \sigma_1, m_4)$ depending on $L, \ \sigma, \ \sigma_1, m_4,$ 
such that for $ \ 1\leq i \leq N,$
\begin{equation}
\label{chto1}
\big|\E (f(X_N)_{ii})-\int_{-2\sigma}^{2\sigma} f(x) \* \frac{1}{2 \pi \sigma^2} \sqrt{ 4 \sigma^2 - x^2}  \* dx\big|  
\leq \frac{Const_1(L, \sigma, \sigma_1, m_4)}{N} \* \|f\|_{C^7([-L,L])} 
\end{equation}

(v) Let $f \in C^8(\R), $
then there exists a constant $Const_2(\sigma, \sigma_1, m_4)$ such that
\begin{align}
\label{chto11}
& \big|\E (f(X_N)_{ii})-\int_{-2\sigma}^{2\sigma} f(x) \* \frac{1}{2 \pi \sigma^2} \sqrt{ 4 \sigma^2 - x^2}  \* dx\big|  \\
& \leq \frac{Const_2(\sigma, \sigma_1, m_4)}{N} \* \|f\|_{8,1,+}, \ \ 1\leq i \leq N, \nonumber
\end{align}
where
\begin{equation}
\label{normasobolev}
\|f\|_{n,1,+}:=\max \left( \int_{\R} (|x|+1) \*|\frac{d^lf}{dx^l}(x)| \* dx, \ 0\leq l \leq n \right).
\end{equation}

(vi) Let $f\in C^6(\R), $
then there exists a constant $Const_3(\sigma, \sigma_1, m_4)$ such that
\begin{equation}
\label{chto12}
\big|\E (f(X_N)_{jk})\big|\leq \frac{Const_3(\sigma, \sigma_1, m_4)}{N} \* \|f\|_{6,1}, \ \ 1\leq j <k\leq N,
\end{equation}
where $\|f\|_{6,1}$ is defined in (\ref{normsob}).
\end{theorem}

\begin{remark}
If the distribution of the entries of $W_N$ does not depend on $N,$ Theorem \ref{thm:main} proves that
$\sqrt{N} \*\left( f(X_N)_{ii} - \E\left( f(X_N)_{ii}\right) \right)$
converges in distribution to the sum of two independent random variables
$\frac{\alpha(f)}{\sigma}\* W_{ii}$ and $N(0, 2\*v^2(f))$  (in the Hermitian case, the second term is $N(0, v^2(f))$),
and for $i\not=j, \
\sqrt{N} \* \left( f(X_N)_{ij} - \E\left( f(X_N)_{ij}\right)\right)$
converges in distribution
to the sum of two independent random variables
$\frac{\alpha(f)}{\sigma} \*W_{ij}$ and $ N(0, d^2(f)),$
where in the Hermitian case $N(0, d^2(f))$ stands for the complex Gaussian random variable with 
with i.i.d real and imaginary parts $N(0, \frac{1}{2}\*d^2(f)).$  
This is exactly the way Theorems \ref{thm:real} and \ref{thm:herm} were formulated and proven in the i.i.d. case in \cite{PRS}.
\end{remark}

\begin{remark}
\label{zamechanie}
If $f:\R\to \R$ coincides on $[-2\sigma-\delta, 2\sigma+\delta]$  with a function  $\phi \in \mathcal{H}_s,$  for some $\delta>0$ and $s>3,$
then Theorem \ref{thm:main} holds for $(f(X_N))_{ij}-\E (fh(X_N))_{ij}, \ 1\leq i,j\leq m,$ where $h\in C^{\infty}_c(\R)$ is defined in (\ref{amerika}).
\end{remark}

If  one requires that the test function $f$ satisfies the same smoothness assumptions as in \cite{PRS}, then the extension of the results of 
\cite{PRS} to the non-i.i.d. setting mostly follows the outline of the proof in \cite{PRS}. 
%The main changes are the extension of the Central
%Limit Theorem for quadratic forms to the non-i.i.d. case which  essentially follows from the arguments of Bai and Yao \cite{BY}, Baik 
%and Silverstein \cite{CDF}, and Benaych-Georges, Guionnet, and Maida \cite{BGM} developed in the i.i.d. setting 
%(see the Appendix on the CLT for quadratic forms) and the 
%extension of the theorem of Bai and Yin (\cite{BYin}, \cite{B}) about the convergence of the largest eigenvalue of $X_N$ to the right edge of the 
%Wigner semicircle law to the non-i.i.d. setting (\ref{offdiagreal}-\ref{diagreal}) (see Proposition \ref{proposition:prop1} in Section \ref{sec:proofs}).
To relax the conditions of Thereoms \ref{thm:real} and \ref{thm:herm} on the test functions, we improve the estimate on the variance of the resolvent 
entries (see Proposition \ref{proposition:prop4}), and employ Proposition \ref{prop:variance}. 

We will denote throughout the paper by $const_i, Const_i, $ various positive constants
that may change from line to line.  Occasionally, we will drop the dependence on $N$ in the notations for the matrix entries.
Typically, we consider in detail only the real symmetric case as the proofs in the Hermitian case are very similar. Some parts of the proofs that are 
almost identical to the arguments in the i.i.d. case will be only sketched.

The rest of the paper is organized as follows. We prove several preliminary results in Section \ref{sec:prelim}, including Proposition 
\ref{prop:variance}. Section \ref{sec:estvar} is devoted to the bounds on the mathematical expectation and variance of the resolvent entries.
Theorem \ref{thm:main} is proved in Section \ref{sec:proofs}.  Finally, we discuss Central Limit Theorem for quadratic forms 
in the Appendix.

\section{ \bf{Preliminary Results}}
\label{sec:prelim}
We start with the following lemma.

\begin{lemma}
\label{lemma:1}
Let $X_N=\frac{1}{\sqrt{N}} W_N$ be a random real symmetric Wigner matrix
(\ref{offdiagreal}), (\ref{diagreal}) such that
the Lindeberg condition (\ref{lind1}) for the fourth moments of the off-diagonal entries
and the Lindeberg condition (\ref{diagreal1}) for the second moments of the diagonal entries
are satisfied.
Then there exists a random real symmetric Wigner matrix $\tilde{W}_N$ and a non-random positive sequence $\epsilon_N\to 0$ as $N \to \infty$ such that
\begin{align}
\label{offdiagrealw}
& \E (\tilde{W}_N)_{jk}=0, \ \V(\tilde{W}_N)_{jk}=\sigma^2, \ 1\leq j<k\leq N, \\
\label{offdiagrealw1}
& \ \sup_{N,j\not=k} \E (\tilde{W}_N)_{jk}^4<\infty, \\
\label{diagrealw1}
& \E (\tilde{W}_N)_{ii}=0, \ \ 1\leq i \leq N, \\
\label{diagrealw2}
& \sup_{i,N} \E |(\tilde{W}_N)_{ii}|^2 < \infty,\\
\label{wtildaw1}
& \sup \left(|(\tilde{W}_N)_{ij}|, \ 1\leq i,j\leq N \right) \leq \epsilon_N\*\sqrt{N}, \\
\label{wtildaw2}
& \P(W_N\not=\tilde{W}_N)\to 0, \ \text{as} \ N\to \infty.
\end{align}
\end{lemma}
An equivalent result holds in the Hermitian case.
\begin{proof}
It follows from (\ref{lind1}) and (\ref{diagreal1})
that there exists a non-random positive sequence $\epsilon_N \to 0$ as $N \to \infty,$ such that
\begin{align}
\label{lind3}
& \frac{1}{N^2\*\epsilon_N^4}\* \sum_{1\leq i<j\leq N} \E\left( |(W_N)_{ij}|^4 \*\indicator{|(W_N)_{ij}|\geq \epsilon_N\*\sqrt{N}}\right) \to 0.\\
\label{li}
&  \frac{1}{N\*\epsilon_N^2}\* \sum_{1\leq i\leq N} \E\left( |(W_N)_{ii}|^2 \*\indicator{|(W_N)_{ii}|\geq \epsilon_N\*\sqrt{N}}\right) \to 0.
\end{align}
One can always choose $\epsilon_N$ in such a way that it goes to zero sufficiently slow.
Define $\bar{W}_N$ by truncating the entries of $W_N$ at the level $\epsilon_N\*\sqrt{N},$ i.e.
\begin{equation}
\label{barw}
(\bar{W}_N)_{ij}= (W_N)_{ij} \* \indicator{|(W_N)_{ij}|\leq \epsilon_N\*\sqrt{N}}.
\end{equation}
It follows from (\ref{lind3}) and (\ref{li}) that 
\begin{equation}
\label{dodo}
\P(W_N\not=\bar{W}_N)\to 0, \ \text{as} \ N\to \infty.
\end{equation}
Let us now fix $i<j$ and consider the off-diagonal entry $(\bar{W}_N)_{ij}.$
We note that 
\begin{align}
\label{mo1}
& \tau_{i,j,N}:=|\E (\bar{W}_N)_{ij}| \leq  \E\left( |(W_N)_{ij}| \*\indicator{|(W_N)_{ij}|\geq \epsilon_N\*\sqrt{N}}\right) \\
\label{mo2}
& \leq \frac{1}{N^{3/2}\*\epsilon_N^3} \E\left( |(W_N)_{ij}|^4 \*\indicator{|(W_N)_{ij}|\geq \epsilon_N\*\sqrt{N}}\right), \\
\label{v1}
& \gamma_{i,j,N}^2:= \E |\bar{W}_N|^2_{ij}-\sigma^2= \E\left( |(W_N)_{ij}|^2 \*\indicator{|(W_N)_{ij}|\geq \epsilon_N\*\sqrt{N}}\right) \\
\label{v2}
& \leq \frac{1}{N \*\epsilon_N^2} \E\left( |(W_N)_{ij}|^4 \*\indicator{|(W_N)_{ij}|\geq \epsilon_N\*\sqrt{N}}\right).
\end{align}
Then we can constract $(\tilde{W}_N)_{ij}$ as a mixture of the random variable 
$(\bar{W}_N)_{ij}$ with weight $1-\frac{\tau_{i,j,N}}{\sqrt{N}\*\epsilon_N} -
\frac{\gamma_{i,j,N}^2}{N\*\epsilon_N^2}
$ and some random 
variable $a_{i,j,N}$ with weight $\frac{\tau_{i,j,N}}{\sqrt{N}\*\epsilon_N}+ \frac{\gamma_{i,j,N}^2}{N\*\epsilon_N^2}$ chosen so that 
\begin{align}
\label{tv1}
& |a_{i,j,N}|\leq \epsilon_N\*\sqrt{N}, \\
\label{tv2}
& \E (\tilde{W}_N)_{ij}=0,\\
\label{tv3}
& \E |(\tilde{W}_N)_{ij}|^2=\sigma^2.
\end{align}
It follows from our construction and (\ref{lind3}) that
\begin{equation}
\label{dor}
\sum_{1\leq i<j\leq N} \P\left((\bar{W}_N)_{ij}\not= (\tilde{W}_N)_{ij}\right) \leq 
\frac{2}{N^2\*\epsilon_N^4}\* \sum_{1\leq i<j\leq N} \E\left( |(W_N)_{ij}|^4 \*\indicator{|(W_N)_{ij}|\geq \epsilon_N\*\sqrt{N}}\right) \to 0.
\end{equation}

The diagonal case $i=j$ can be treated in a similar way.
We write
\begin{align}
\label{Mo1}
& \tau_{i,i,N}:=|\E (\bar{W}_N)_{ii}| \leq \E  \left( |(W_N)_{ii}| \*\indicator{|(W_N)_{ii}|\geq \epsilon_N\*\sqrt{N}}\right) \\
\label{Mo2}
& \leq \frac{1}{\sqrt{N}\*\epsilon_N} \E\left( |(W_N)_{ii}|^2 \*\indicator{|(W_N)_{ii}|\geq \epsilon_N\*\sqrt{N}}\right).
\end{align}
One then constructs $(\tilde{W}_N)_{ii}$ as a mixture of the random variable 
$(\bar{W}_N)_{ii}$ with weight $1-\frac{\tau_{i,i,N}}{\sqrt{N}\*\epsilon_N}$ 
and some random variable $a_{i,i,N}$ with weight $\frac{\tau_{i,i,N}}{\sqrt{N}\*\epsilon_N}$ chosen so that
\begin{align}
\label{tv4}
& |a_{i,i,N}|\leq \epsilon_N\*\sqrt{N}, \\
\label{tv5}
& \E (\tilde{W}_N)_{ii}=0.
\end{align}
Then
\begin{equation}
\label{dor1}
\sum_{1\leq i \leq N} \P\left((\bar{W}_N)_{ii}\not= (\tilde{W}_N)_{ii}\right) \leq 
\frac{1}{N\*\epsilon_N^2} \E\left( |(W_N)_{ii}|^2 \*\indicator{|(W_N)_{ii}|\geq \epsilon_N\*\sqrt{N}}\right)\to 0,
\end{equation}
as $N \to \infty.$

It follows from (\ref{dodo}), (\ref{dor}), and (\ref{dor1}) that (\ref{wtildaw2}) is satisfied. The equations (\ref{offdiagrealw}) and (\ref{diagrealw1})
follow from (\ref{tv2}), (\ref{tv3}), and (\ref{tv5}). The estimates (\ref{offdiagrealw1}) and (\ref{diagrealw2}) follow from the construction.
\end{proof}
The proof of the next result is very similar to the proof of Lemma \ref{lemma:1} and is left to the reader.
\begin{lemma}
\label{lemma:2}
Let $W_N$ be a random real symmetric Wigner matrix
(\ref{offdiagreal}), (\ref{diagreal}), and
let (\ref{LIND1}) is satisfied for $1\leq i\leq m,$ where $m$ is some fixed positive integer.
Then there exists a random real symmetric Wigner matrix $T_N$ and a non-random positive sequence $\epsilon_N\to 0$ as $N \to \infty$ such that
\begin{align}
\label{mir0}
& (T_N)_{jk}= (W_N)_{jk}, \  m+1\leq j,k\leq N, \\
\label{mir1} 
& \P ((T_N)_{ik}= (W_N)_{ik}, 1\leq i\leq m, \ 1\leq k\leq N)\to 1, \ \text{as} \ N \to \infty, \\
\label{mir2}
& \E (T_N)_{ik}=0, \ 1\leq i\leq m, \  1\leq k \leq N, \\
\label{mir3}
& \V(T_N)_{ik}=\sigma^2, \ i\not=k, 1\leq i\leq m, \  1\leq k \leq N, \ \sup_{1\leq i \leq m, N} \V(T_N)_{ii} <\infty, \\
\label{mir4}
& \ \sup_{N,i\not=k: 1\leq i\leq m, \ 1\leq k\leq N} \E (T_N)_{ik}^4<\infty, \\
\label{mir5}
& \sup_{1\leq i\leq m, 1\leq k \leq N}  |(T_N)_{ik}|  \leq \epsilon_N\*N^{1/4}.
\end{align}
\end{lemma}

The next Proposition is essentially due to Bai and Yin (see e.g. \cite{BYin}, \cite{B}).
\begin{proposition}
\label{proposition:prop1}
Let $X_N=\frac{1}{\sqrt{N}} W_N$ be a random real symmetric (Hermitian) Wigner matrix
(\ref{offdiagreal}), (\ref{diagreal}) (respectively (\ref{offdiagherm1}-\ref{diagherm})  such that
the Lindeberg type condition (\ref{lind1}) for the fourth moments of the off-diagonal entries and the
Lindeberg type condition (\ref{diagreal1}) for the second moments of the diagonal entries
are satisfied.
Then 
\begin{equation}
\label{byresult}
\|X_N\| \to 2\*\sigma
\end{equation}
in probability as $N \to \infty.$
\end{proposition}
\begin{remark}
Bai and Yin (\cite{BYin}, \cite{B}) considered the i.i.d. case and 
proved the almost sure convergence.  However, convergence in probability is enough for our purposes.
\end{remark}
\begin{proof}
Because of Lemma \ref{lemma:1}, it is enough to prove (\ref{byresult}) for $\tilde{W}_N.$ Moreover, we can modify 
$\tilde{W}_N$ by making all diagonal entries equal to zero.  Clearly this changes the norm of $\tilde{W}_N$ at most by $\epsilon_N.$
The proof uses the Method of Moments.
It is enough to show that there exists a sequence $k_N, \ N\geq 1, $ such that
\begin{equation}
\label{ust1}
\frac{k_N}{\log N} \to \infty, \ \ \frac{\epsilon_N^{1/3} \*k_N}{\log N} \to 0, \ \text{as} \ N\to \infty,
\end{equation}
where $\epsilon_N$ is the same as in Lemma \ref{lemma:1}, and for any constant $z>2\sigma$
\begin{equation}
\label{ust2}
\sum_N \frac{\Tr\left((\tilde{W}_N/\sqrt{N})^{2\*k_N}\right)}{z^{2\*k_N}} <\infty.
\end{equation}
The proof of (\ref{ust2}) in (\cite{BYin}) is combinatorial in nature and does not use the fact that the entries are identically distributed.
By Markov inequality, it follows from (\ref{ust2}) that
\begin{equation*}
\sum_N \P(\|\tilde{W}_N/\sqrt{N}\|\geq z)<\infty,
\end{equation*}
for any fixed $z>2\*\sigma.$ Therefore, by Borel-Cantelli lemma, we have
\begin{equation*}
\P(\|\tilde{W}_N/\sqrt{N}\|\geq z \ i.o.)=0,
\end{equation*}
which together with the Semicircle Law implies that
\begin{equation*}
\|\tilde{W}_N/sqrt{N}\|\to 2\sigma \ \text{a.s.}
\end{equation*}
\end{proof}
%%%%%%%%%%%%%%%%%%%%%%%%%%%%%%%%%%%%%%%%%%%%%%%%%%%%%%%%%%%%%%%%%%%%%%%%%%%%%%%%%%%%%%%%%%%%%%%%%%%%%%%%%%%%%%%%%%%%%%%%%%%%%%%%%%%%%%%%%%%%%%%%%%%%%%%%%%%
%%%%%%%%%%%%%%%%%%%%%%%%%%%%%%%%%%%%%%%%%%%%%%%%%%%%%%%%%%%%%%%%%%%%%%%%%%%VARIANCE BOUNDS%%%%%%%%%%%%%%%%%%%%%%%%%%%%%%%%%%%%%%%%%%%%%%%%%%%%%%%%%%%%%%%%%
%%%%%%%%%%%%%%%%%%%%%%%%%%%%%%%%%%%%%%%%%%%%%%%%%%%%%%%%%%%%%%%%%%%%%%%%%%%%%%%%%%%%%%%%%%%%%%%%%%%%%%%%%%%%%%%%%%%%%%%%%%%%%%%%%%%%%%%%%%%%%%%%%%%%%%%%%%%

The rest of this section is devoted to the bounds on 
$ \V [\int_{-\infty}^{\infty} f(x) \* \mu(dx, \omega)], $ where $\mu(dx, \omega)$ is a random measure on $(\R, \mathcal{B})$ 
and $ \mathcal{B}$ is the Borel $\sigma$-algebra on $\R,$
provided one can control
$\V [\int_{-\infty}^{\infty} \Im \frac{1}{z-x} \* \mu(dx, \omega)]$ for $\Im z \not=0.$  We follow the ideas of Proposition 1 in \cite{Sh} and 
Proposition 3.5 in \cite{Joh}. In particular, our computations below are close to those in \cite{Sh}, where
$\mu(dx, \omega)$ was taken to be the empirical spectral distribution of a random  matrix. 

Let $(\Omega, \mathcal{F})$ be a measurable space, and  
$(\Omega', \mathcal{F'}, \mathcal{P})$ be a probability space such that $\Omega'=\R \times \Omega,$ and $\mathcal{F'}$ is generated by 
$\mathcal{B} \times \mathcal{F}.$  We denote an elementary outcome by $\omega'=(x, \omega) \in \R \times \Omega, $ and consider a random variable
%\begin{equation*}
$ X(\omega')=x.$
%\end{equation*}
When it does not lead to ambiguity, we will denote the sub-algebra $\{ \R\times D, \ D\in \mathcal{F} \}$ by $\mathcal{F}.$
Let us denote by $\mu(B, \omega), \ B \in \mathcal{B}, \ \omega \in \Omega,$ a  regular conditional distribution for $X$ given $\mathcal{F},$ i.e.
\begin{align}
\label{rcd1}
& \text{For each} \ \ B\subset \R, \ B \in \mathcal{B}, \ \omega \to \mu(B, \omega) \ \text{is a version of} \ \mathcal{P}(X \in B|\mathcal{F}). \\
\label{rcd2}
& \text{For a.e.} \ \omega, \ B\to \mu(B, \omega) \ \text{is a probability measure on} \ (\R, \mathcal{B}).
\end{align}
Such regular conditional distribution for $X$ always exists (see e.g. \cite{Dur}).  In particular, if $f:\R\to \C$ is such that
\begin{equation}
\label{kon}
\E |f(X)|<\infty, 
\end{equation}
then
\begin{equation*}
\E(f(X)|\mathcal{F})= \int_{-\infty}^{+\infty} f(x) \* \mu(dx, \omega) \ a.s.
\end{equation*}

The following proposition holds.

\begin{proposition}
\label{prop:variance}
Let $E |X| < \infty, \ s>\frac{1}{2},$ and  $f \in \mathcal{H}_s, $  where $\mathcal{H}_s$ is defined in (\ref{sobolev}). Then
\begin{align*}
& \V [ \int f(x) \* \mu(dx, \omega)]=\V [ \E (f(X)|\mathcal{F})] \\
& \leq Const_s \* \|f\|_s^2 \* \int_0^{\infty} dy\* e^{-y}\* y^{2s-1} \* \int_{-\infty}^{\infty} \* dx
V [\int_{\infty}^{+\infty} \Im \frac{1}{t-x-iy}\* \* \mu(dt, \omega)]. \\
\end{align*}
where $Const_s$ is some absolute constant that depends only on $s.$
\end{proposition}

\begin{proof}
Since $s>\frac{1}{2}, $ it follows from (\ref{sobolev}) that $\hat{f} \in L^1(\R)$ which implies that  
$f\in C_0(\R),$ the space of continuous functions vanishing at infinity.
In particular, (\ref{kon}) holds and $ \E (f(X)|\mathcal{F})$ is well defined. Since  $\E (e^{i\*k\*X}|\mathcal{F}), \ k\in \R,$ is $L^1$ continuous 
family of bounded random variables, one can write
\begin{equation}
\label{ap1}
\E (f(X)|\mathcal{F}) = \frac{1}{\sqrt{2\pi}} \* \int_{-\infty}^{\infty} \hat{f}(k) \* \E (e^{i\*k\*X}|\mathcal{F}) \* dk.
\end{equation}
Then
\begin{equation}
\label{ap2}
\V [ \E (f(X)|\mathcal{F})]= \frac{1}{2\pi} \* \int_{-\infty}^{\infty} \*\int_{-\infty}^{\infty} \* \hat{f}(k_1) \* \overline{\hat{f}(k_2)} \* C(k_1, k_2) \*dk_1dk_2,
\end{equation}
where
\begin{equation}
\label{cov}
C(k_1, k_2)=\Cov \left(\E (e^{i\*k_1\*X}|\mathcal{F}), \E (e^{i\*k_2\*X}|\mathcal{F})\right).
\end{equation}
One can rewrite the r.h.s. of (\ref{ap2}) as
\begin{equation}
\label{ap3}
\frac{1}{2\pi} \* \int_{-\infty}^{\infty} \*\int_{-\infty}^{\infty} \* \hat{f}(k_1) \*(1+2|k_1|)^s \* \overline{\hat{f}(k_2)}\*(1+2|k_2|)^s 
\* K(k_1, k_2) \*dk_1dk_2,
\end{equation}
where
\begin{equation}
\label{ap4}
K(k_1, k_2)= C(k_1, k_2) \*(1+2|k_1|)^{-s} \* \*(1+2|k_2|)^{-s}.
\end{equation}
Therefore,
\begin{equation}
\label{ap5}
\V [ \E (f(X)|\mathcal{F})] \leq  \frac{1}{2\pi}  \* \|f\|^2_s \* \|K\|,
\end{equation}
where $\|K\| $ denotes the operator norm of the integral operator
\begin{equation*}
K:L^2(\R)\to L^2(\R), \ (K\*g)(x)=\int_{-\infty}^{\infty} K(x,y)\*g(y)\*dy.
\end{equation*}
It follows from (\ref{cov}) and (\ref{ap4}) that $K$ is a non-negative definite operator.  Since $C(k_1, k_2)$ is a bounded continuous function
on $\R^2,$ the operator $K$ is trace class and
\begin{equation}
\label{ap6}
\|K\| \leq \Tr K=\int_{-\infty}^{\infty} K(u,u) \* du.
\end{equation}
Thus,
\begin{equation}
\label{ap8}
\V [ \E (f(X)|\mathcal{F})] \leq \frac{1}{2\pi} \* \|f\|^2_s \* \int_{-\infty}^{\infty} C(k,k) \* (1+2\*|k|)^{-2s} \* dk.
\end{equation}
Let us fix $z=x+iy, \ y\not=0,$ and consider $\Im \frac{1}{\lambda-z}$ as a function of $\lambda.$ Its Fourier tranform 
is given by $\frac{\sqrt{\pi}}{\sqrt{2}}\* e^{-|ky|-ikx}.$  Therefore,
\begin{equation}
\label{ap9}
V [\E ( \Im(X-x-iy)^{-1}|\mathcal{F})]= \frac{1}{4}\* \int_{-\infty}^{\infty} \*\int_{-\infty}^{\infty} \* e^{-(|k_1|+|k_2|)\* |y|} \* 
e^{i\*(k_2-k_1)\*x} \* C(k_1, k_2) \*dk_1dk_2.
\end{equation}
Formally, taking into account
\begin{equation*}
\int_{-\infty}^{\infty} e^{i\*(k_2-k_1)\*x} \* dx = 2\*\pi\* \delta(k_2-k_1),
\end{equation*}
we obtain
\begin{equation}
\label{ap10}
\int_{-\infty}^{\infty} V [\E ( \Im(X-x-iy)^{-1}|\mathcal{F})]\*dx = \frac{\pi}{2}\* 
\int_{-\infty}^{\infty} \*\int_{-\infty}^{\infty} \* e^{-(|k_1|+|k_2|)\* |y|} \* 
\delta(k_2-k_1)\* C(k_1, k_2) \*dk_1dk_2.
\end{equation}
Since
\begin{equation*}
\int_0^{+\infty} dy\* e^{-y}\* y^{2s-1} \* e^{-2\*|k|\* |y|}= \Gamma(2s)\* (1+2\*|k|)^{-2s},
\end{equation*}
we conclude that 
\begin{equation}
\label{ap11}
\int_0^{\infty} dy\* e^{-y}\* y^{2s-1} \* \int_{-\infty}^{\infty} 
V [\E ( \Im(X-x-iy)^{-1}|\mathcal{F})] \* dx = \frac{\pi}{2}\* \Gamma(2s) \* \int_{-\infty}^{\infty} C(k,k) \* (1+2\*|k|)^{-2s} \* dk.
\end{equation}
The bound on $\V [ \E (f(X)|\mathcal{F})]$ in Proposition \ref{prop:variance} now follows from (\ref{ap8}) and (\ref{ap10}).

To make the steps (\ref{ap10}-\ref{ap11}) rigorous,
we first restrict integration in (\ref{ap10}) to $[-A, A],$ and then let $A\to \infty.$
It follows from (\ref{ap9}) that
\begin{align}
\label{ap100}
& \int_{-A}^A  V [\E ( \Im(X-x-iy)^{-1}|\mathcal{F})]\*dx \\
\label{ap101}
& = \frac{1}{2}\* 
\int_{-\infty}^{\infty} \*\int_{-\infty}^{\infty} \* e^{-(|k_1|+|k_2|)\* |y|} \* 
\frac{\sin(A\*(k_2-k_1))}{k_2-k_1}\* C(k_1, k_2) \*dk_1dk_2.
\end{align}
Multiplying (\ref{ap101}) by $e^{-y}\* y^{2s-1}$ and integrating over $y \in (0, +\infty)$, we obtain
\begin{align}
\label{ap102}
& \int_0^{+\infty} dy \* e^{-y}\* y^{2s-1}  \*  \int_{-A}^A  dx \* V [\E ( \Im(X-x-iy)^{-1}|\mathcal{F})] \\
\label{ap103} 
& =  \frac{1}{2}\*\Gamma(2s) \*
\int_{-\infty}^{\infty} \*\int_{-\infty}^{\infty} \* (1+|k_1|+|k_2|)^{-2s} \* 
\frac{\sin(A\*(k_2-k_1))}{k_2-k_1}\* C(k_1, k_2) \*dk_1dk_2.
\end{align}
We note that the integrand in (\ref{ap103}) is absolutely integrable over $\R^2$ for $s>\frac{1}{2},$  so the last step is justified by the 
Foubini theorem.
Since $E |X|<\infty, $  it follows from (\ref{cov}) that the kernel $C(k_1,k_2)$ has bounded continuous first partial derivatives 
(see Lemma \ref{lemma:cov}) below. We split the integral in (\ref{ap103}) into two, over $$ S:=\{(k_1,k_2): |k_2-k_1|< A^{-\epsilon} \}$$ and over 
$\R^2\setminus S.$
For $(k_1,k_2) \in S, $ we use 
\begin{equation*}
|C(k_1, k_2)\*(1+|k_1|+|k_2|)^{-2s}  -C(k_1, k_1)\* (1+2\*|k_1|)^{-2s}   | \leq const \* |k_2-k_1| \* (1+2\*|k_1|)^{-2s}
\end{equation*}
which implies that the integral over $S$ equals to
\begin{equation}
\label{ap199}
\frac{1}{2}\*\Gamma(2s) \* \int_{-A^{-\epsilon}}^{A^{-\epsilon}}  \* dt  \* \frac{\sin(A\*t)}{t} \*\int_{-\infty}^{\infty} \* dk 
C(k,k) \* (1+2\*|k|)^{-2s}  + o(1).
\end{equation}
where we made the change of variables $(k_1, k_2)\to (t=k_2-k_1, k=k_1).$

To estimate the integral over $\R^2\setminus S,$  we restrict our attention to the quadrant $k_1\geq 0, k_2 \geq 0.$ The other three cases are similar.
Denote $C_1(t,u)=C(k_1, k_2), $ where $u=k_1+k_2$ and $t=k_2-k_1.$
We have to estimate
\begin{equation}
\label{ap200}
\int_0^{\infty}  \* du \*\int_{A^{-\epsilon}}^u \* dt \* (1+u)^{-2s} \* \frac{\sin(A\*t)}{t}\* C_1(t,u) .
\end{equation}
Integrating by parts with respect to $t,$ we obtain
\begin{align}
\label{ap201}
& \int_{A^{-\epsilon}}^u \* \frac{\sin(A\*t)}{t}\* C_1(t,u) \* dt
= \int_{A^{-\epsilon}}^u \* \frac{\cos(A\*t)}{A} \* \left(\frac{\partial C_1(t,u)}{\partial t} \* \frac{1}{t} - C_1(t,u)\* \frac{1}{t^2} \right) \\
\label{ap202}
& - \frac{\cos(A\*t)}{A \* t} \* C_1(t,u)\big|^u_{A^{-\epsilon}}.
\end{align}
It is not difficult to see that the r.h.s. of (\ref{ap201}) is bounded in absolute value by $ const \* \frac{1}{A}\* (|\log u| + A^{\epsilon})$ and
(\ref{ap202}) is bounded in absolute value by $ const \* A^{-1+ \epsilon}.$  Therefore, ther integral over $\R^2\setminus S $ goes to zero as 
$A \to \infty.$

Finally, we note that the term in (\ref{ap199}) converges to
\begin{equation}
\Gamma(2s) \* \frac{\pi}{2}\* \int_{-\infty}^{\infty} 
C(k,k) \* (1+2\*|k|)^{-2s}  \* dk.
\end{equation}
This finishes the proof of Proposition \ref{prop:variance}, modulo Lemma \ref{lemma:cov} below.
\end{proof}
In the proof of Proposition \ref{prop:variance}, we used the fact that $C(k_1,k_2),$ defined in (\ref{cov}), has continuous bounded first partial 
derivatives.  This is the statement of the following lemma.
\begin{lemma}
\label{lemma:cov}
Let $\E |X| <\infty, $ and $C(k_1, k_2)$ be defined as in (\ref{cov}).  Then $C(k_1, k_2)$ has continuous bounded first partial 
derivatives. 
\end{lemma}
\begin{proof}
We recall that
\begin{align*}
& C(k_1, k_2)=\Cov \left(\E (e^{i\*k_1\*X}|\mathcal{F}), \E (e^{i\*k_2\*X}|\mathcal{F})\right) \\
& = \E \left( \E (e^{i\*k_1\*X}|\mathcal{F}) \* \E (e^{i\*k_2\*X}|\mathcal{F}) \right) - \E (e^{i\*k_1\*X}) \* \E (e^{i\*k_1\*X}).
\end{align*}
It follows from the Lebesgue dominated convergence theorem (for conditional expectations), that  
\begin{equation}
\label{ap300}
\frac{\partial C(k_1, k_2)}{\partial k_1}= i \* \Cov \left(\E (X \* e^{i\*k_1\*X}|\mathcal{F}), \E (e^{i\*k_2\*X}|\mathcal{F})\right).
\end{equation}
Applying the Lebesgue dominated convergence theorem one more time, we obtain that $\frac{\partial C(k_1, k_2)}{\partial k_1}$
is a bounded continuous function.
\end{proof}

%%%%%%%%%%%%%%%%%%%%%%%%%%%%%%%%%%%%%%%%%%%%%%%%%%%%%%%%%%%%%%%%%%%%%%%%%%%%%%%%%%%%%%%%%%%%%%%%%%%%%%%%%%%%%%%%%%%%%%%%%%%%%%%%%%%%%%%%%%%%%%%%%%%%%%%%%%%
%%%%%%%%%%%%%%%%%%%%%%%%%%%%%%%%%%%%%%%%%%%%%%%%%%%%%RESOLVENT ENTRIES%%%%%%%%%%%%%%%%%%%%%%%%%%%%%%%%%%%%%%%%%%%%%%%%%%%%%%%%%%%%%%%%%%%%%%%%%%%%%%%%%%%%%
%%%%%%%%%%%%%%%%%%%%%%%%%%%%%%%%%%%%%%%%%%%%%%%%%%%%%%%%%%%%%%%%%%%%%%%%%%%%%%%%%%%%%%%%%%%%%%%%%%%%%%%%%%%%%%%%%%%%%%%%%%%%%%%%%%%%%%%%%%%%%%%%%%%%%%%%%%%

\section{ \bf{Mathematical Expectation and Variance of Resolvent Entries}}
\label{sec:estvar}
This section is devoted to the estimates of the mathematical expectation and the variance of the resolvent entries.  
For $z\in \C \setminus \R,$ we denote the resolvent of $X_N$ by
\begin{equation}
\label{resolventa}
R_N(z):=(z\*I_N-X_N)^{-1}.
\end{equation}
If it does not lead to ambiguity, we will use the shorthand notation $R_{ij}(z)$ for $(R_N(z))_{ij}, \ 1\leq i,j\leq N.$
We start with the following proposition.

\begin{proposition}
\label{proposition:prop2}

Let $X_N=\frac{1}{\sqrt{N}} W_N$ be a random real symmetric (Hermitian) Wigner matrix
(\ref{offdiagreal}), (\ref{diagreal}) (respectively (\ref{offdiagherm1}-\ref{diagherm})).
Then
\begin{align}
\label{odinnadtsat100}
& \E R_{ii}(z) =g_{\sigma}(z) + O \left( \frac{1}{|\Im z|^6 \*N}\right),  \\
& \text{uniformly on bounded subsets of} \  \C\setminus \R,  \nonumber \\
\label{odinnadtsat101}
& \E R_{ij}(z)=O \left( \frac{P_5(|\Im z|^{-1})}{N}\right), \  1\leq i\not=j \leq N,\  \text{uniformly on} \ \C\setminus \R, \\
\label{odinnadtsat102}
& \V R_{ij}(z) = O \left( \frac{P_6(|\Im z|^{-1})}{N}\right), \  1\leq i, j\leq N, \ \text{uniformly on} \ \C\setminus \R. 
\end{align}
where we denote by $P_l(x), \ l\geq 1, $ a polynomial of degree $l$ with fixed positive coefficients.

If, in addition,  
\begin{equation*}
\sup_{i\not=j, N} \E |(W_N)_{ij}|^5 <\infty, \ \ \sup_{i, N} \E |(W_N)_{ii}|^3 <\infty,
\end{equation*}
then
\begin{equation}
\label{odinnadtsat103}
\E R_{ij}(z)=O \left( \frac{P_9(|\Im z|^{-1})}{N^{3/2}}\right), \  1\leq i\not=j \leq N, \ \text{uniformly on} \ \C\setminus \R.
\end{equation}
\end{proposition}
This proposition is the extension of Proposition 3.1 in \cite{PRS} to the non-i.i.d. case. Since the proofs of 
(\ref{odinnadtsat100}-\ref{odinnadtsat103}) are very similar to the proofs given in Proposition 3.1 in Section 2 of \cite{PRS}, we leave the details to 
the reader.

The next proposition is instrumental in extending Theorem \ref{thm:main} to the test functions from $\mathcal{H}_s$ for $s>3.$
Our goal is to obtain an upper bound on $\V [(R_N)_{ij}(z)]$ which is integrable with respect to $x=\Re z$ over the real line for $\Im z \not=0.$
\begin{proposition}
\label{proposition:prop4}

Let $X_N=\frac{1}{\sqrt{N}} W_N$ be a random real symmetric Wigner matrix
(\ref{offdiagreal}), (\ref{diagreal}) such that
the condition (\ref{LIND1}) is satisfied for some fixed $m\geq 1.$
Then there exists a random real symmetric Wigner matrix $T_N$ and a non-random positive sequence $\epsilon_N\to 0$ as $N \to \infty$ such that
the properties (\ref{mir0}-\ref{mir5}) from Lemma \ref{lemma:2}  are satisfied and, in addition, 

\begin{equation}
\label{odinnadtsat300}
\V [(G_N)_{ij}(z)] = O \left( \frac{ (\E \|G_N(z)\|^2) \* P_4(|\Im z|^{-1})}{N}\right) +O \left( \frac{ (\E \|G_N(z)\|^{3/2}) 
\* P_4(|\Im z|^{-1})}{N}\right),  
\end{equation}
$ 1\leq i \leq m,\ 1\leq j\leq N,$ uniformly on $ \C\setminus \R,$ 
where 
$ G_N(z):=\left(z\*I_N-\frac{1}{\sqrt{N}}\* T_N\right)^{-1}$
is the resolvent of $\frac{1}{\sqrt{N}}\* T_N.$ 
\end{proposition}
An equivalent results holds in the Hermitian case.
\begin{proof}
The existence of a Wigner random matrix $T_N$ that satisfies (\ref{mir0}-\ref{mir5}) follows from Lemma \ref{lemma:2}.  All is left to us is to show that
(\ref{odinnadtsat300}) holds.  Since $\P(X_N=T_N) \to 1$ as $N\to \infty, $ we can assume, without loss of generality, that $T_N=X_N.$ 

Let $L$ be a positive constant that will be later chosen to be sufficiently large depending on $\sigma, \sigma_1, $ and $m_4.$
We note that if
\begin{equation}
\label{40}
\frac{1}{|\Im z|^4\*N} \geq \frac{1}{L}
\end{equation}
then 
\begin{equation}
\label{41}
\V [(R_N)_{ij}(z)] \leq \E \|R_N(z)\|^2 \leq \frac{L \* \E \|R_N(z)\|^2}{|\Im z|^4\*N}.
\end{equation}
Thus, (\ref{40}) implies (\ref{odinnadtsat300}).

Now, let us assume that
\begin{equation}
\label{42}
\frac{1}{|\Im z|^4\*N} < \frac{1}{L}.
\end{equation}
One can rewrite (\ref{42}) as
\begin{equation}
\label{43}
|\Im z| > \frac{L^{1/4}}{N^{1/4}}.
\end{equation}

Let us fix $1\leq i,j \leq m.$ 
%Using the resolvent identity
%\begin{equation}
%\label{resident}
%(zI - A_2)^{-1} = (zI - A_1)^{-1} - (zI - A_1)^{-1}(A_1 - A_2) (zI - A_2)^{-1}, 
%\end{equation}
%with $A_2=X_N,$ and $A_1=0,$ we obtain
Then
\begin{equation}
\label{odin}
z\* \E R_{ij}(z)  = \delta_{ij}+\sum_{k=1}^N \*\E(X_{ik}\*R_{kj}(z)).
\end{equation}
To estimate $\E(X_{ik}\*R_{kj}(z)),$ we use the decoupling formula (see e.g. (i) in Section 2 in \cite{KKP} and Proposition 3.1 in \cite{LytPastur}).
Let $\xi$ be a real random variable with $p+2$ finite moments, and $\phi$ a real-valued function with $p+1$ continuous and bounded derivatives.
Then
\begin{equation}
\label{decouple} 
\E(\xi \phi(\xi)) = \sum_{a=0}^p \frac{\kappa_{a+1}}{a!} \E(\phi^{(a)}(\xi)) + \epsilon_{p+1},  
\end{equation}
where $\kappa_a$ are the cumulants of $\xi$, 
\begin{equation}
\label{bilet}
|\epsilon_{p+1}| \leq C \sup_t \big| \phi^{(p+1)}(t) \big| \E(|\xi|^{p+2}),
\end{equation}
and $C$ depends only on $p$. 
Moreover, as follows from the proof of Proposition 3.1 in \cite{LytPastur}, if $supp(\xi)\subset [-K,K]$ then the supremum on the r.h.s. of 
(\ref{bilet}) can be taken over $t \in  [-K,K].$

The derivative of $R_{kl}$ with respect to $X_{pq}$, for $p \not = q$ is given by 
\begin{equation}
\label{vecher1}
\frac{\partial R_{kl}}{\partial X_{pq}} = R_{kp}\*R_{ql} +R_{kq} \*R_{pl}.
\end{equation}
For $p =q$ the derivative is given by
\begin{equation}
\label{vecher2}
\frac{\partial R_{kl}}{\partial X_{pp}} = R_{kp}\*R_{pl} .
\end{equation}

Applying (\ref{decouple}-\ref{vecher2}) to the term $\E(X_{ik}\*R_{kj})$ in (\ref{odin}), we obtain the following Master equation
\begin{align}
\label{dva1}
& z \* \E R_{ij}(z) = \delta_{ij} +\sigma^2 \* \E [ R_{ij}(z)\*\tr_N R_N(z)]+
\frac{\sigma^2}{N}\E [(R_N(z)^2)_{ij}]  \\
\label{dva2}
& -\frac{2\*\sigma^2}{N}\*\E[R_{ii}(z)\*R_{ij}(z)] + r_N \\
\label{dva3}
& =\delta_{ij} +\sigma^2 \* \E [R_{ij}(z)\*\tr_N R_N(z)] + r_N + O\left(\frac{\E \|R_N(z)\|^2}{N}\right),
\end{align}
where $r_N$ contains the third cumulant term corresponding to $p=2$ in (\ref{decouple}),
and the error due to the truncation of the decoupling formula (\ref{decouple}) at $p=2.$  For $k=i,$
we truncate the decoupling formula (\ref{decouple}) at $p=0.$

We will need the following lemma.

\begin{lemma}
\label{Lemma10}
The following two bounds hold.
\begin{equation}
\label{202}
\Cov(R_{ij}(z), \tr_N R_N(z))=O\left(\frac{P_2(|\Im z|^{-1})\* \E \|R_N(z)\|^{3/2}}{N}\right),
\end{equation}
uniformly in $z\in\C\setminus \R.$
\begin{equation}
\label{203}
r_N=O\left(\frac{P_2(|\Im z|^{-1})\*\E \|R_N(z)\|^2}{N}\right),
\end{equation}
uniformly in $z$ satisfying (\ref{43}), where $L$ is an arbitrary fixed positive number.
\end{lemma}

\begin{proof}
The bound (\ref{202}) follows from the first of the two bounds on the variance of the trace of the resolvent in 
Proposition 2 of \cite{Sh}.  It should be mentioned that the bound is valid provided the second
moments of the diagonal entries are uniformly bounded and the fourth moments
of the off-diagonal entries are also uniformly bounded (\cite{Sh1}).

To prove the bound (\ref{203}), one has to study the third cumulant term that corresponds to  $p=2$ in the decoupling formula (\ref{decouple}) for 
$k\not=i$ and the error terms due to the truncation of (\ref{decouple}) at $p=2$ for $k\not=i$ and at $p=0$ for $k=i.$

The third cumulant term gives
\begin{align}
& \frac{1}{2!\*N^{3/2}} [ 4\* \E (\sum_{k:k\not=i} \kappa_3((W_N)_{ik})\* R_{ij}\*R_{ik}\*R_{kk}) + 
2\*\E (\sum_{k:k\not=i} \kappa_3((W_N)_{ik})\* R_{ii}\*R_{kk}\*R_{kj}) \nonumber \\
&+ 2\*\E (\sum_{k:k\not=i} \kappa_3((W_N)_{ik})\* (R_{ik})^2 \* R_{jk}) ],\nonumber
\end{align}
where $\kappa_3((W_N)_{ik})$ denotes the third cumulant of $(W_N)_{ik}.$
Since $|\kappa_3((W_N)_{ik})|\leq const(m_4), $
\begin{equation}
\label{solnce}
\sum_k |R_{ik}|^2 \leq \|R_N(z) \|^2, \ \text{and} \ |R_{pq}|(z) \leq \|R_N(z)\| \leq \frac{1}{|\Im z|},
\end{equation}
one observes that the third cumulant term can be bounded in absolute value by
\begin{equation*}
O\left(\frac{\E \|R_N(z)\|^2}{|\Im z| \*N}\right).
\end{equation*}

To estimate the error term due to the truncation of (\ref{decouple}) at $p=2$ for $k\not=i,$ we have to consider finitely many sums of the following form
\begin{equation}
\label{trun}
N^{-2} \E  \left (  \sum_{k:k\not=i} \sup |R^{(1)}_{ab}|\*|R^{(2)}_{cd}|\*|R^{(3)}_{ef}|\*|R^{(4)}_{pq}| \right),
\end{equation}
where $a,b,c,d,e,f,p,q,s \in \{i,k,j\}, \ $ the supremum in (\ref{trun}) is considered over all possible resolvents 
$R^{(l)}= (z-X_N^{(l)})^{-1}, \ l=1,\ldots 4$ of 
rank two perturbations $X_N^{(l)}=X_N +x\*E_{ik} $ of $X_N$ with  $(E_{ik})_{jh}=\delta_{ij}\*\delta_{kh} +
\delta_{ih}\*\delta_{kj}. $
Since 
\begin{equation*}
|X_{ik}|\leq \epsilon_N\*N^{-1/4}, \ k\not=i, \ \ \epsilon_N\to 0 \ \text{as} \ N\to \infty,
\end{equation*}
by (\ref{mir5}), we can restrict $x$ in the supremum in (\ref{trun}) to $|x|\leq \epsilon_N \*N^{-1/4}.$
Then
\begin{equation*}
R^{(l)}_N(z)= (z\*I_N-X_N^{(l)})^{-1}= (z\*I_N-X_N +x\*E_{ik})^{-1}=(I_N + R_N(z)\*x\*E_{ik})^{-1}\*R_N(z).
\end{equation*}
Since by taking into account (\ref{43})
\begin{equation*}
\|R_N(z)\*x\*E_{ik}\|\leq \frac{1}{|\Im z|}\* \epsilon_N \*N^{-1/4}\leq \frac{N^{1/4}}{L^{1/4}}\* \epsilon_N \*N^{-1/4}=\epsilon_N\*L^{-1/4}=o(1),
\end{equation*}
we have 
\begin{equation*}
\|R^{(l)}_N(z)\|\leq \|R_N(z)\|\*(1+o(1)),
\end{equation*}
and we obtain that the expression in (\ref{trun}) can be bounded from above by $O\left(\frac{\E \|R_N(z)\|^4}{N}\right).$  It follows from 
\begin{equation}
\label{resbound}
\|R_N(z)\|=\frac{1}{dist(z, Sp(X_N))}\leq |\Im (z)|^{-1}.
\end{equation}
that one can write the upper bound as $O\left(\frac{\E \|R_N(z)\|^2}{|\Im z|^2 \*N}\right).$

To estimate the error term due to the truncation of (\ref{decouple}) at $p=0$ for $k=i,$ one proceeds in a similar manner.
Lemma \ref{Lemma10} is proven
\end{proof}
The rest of the proof of Proposition \ref{proposition:prop4} is similar to the proof of (3.3) in \cite{PRS}.  The details are left to the reader.
\end{proof}
%%%%%%%%%%%%%%%%%%%%%%%%%%%%%%%%%%%%%%%%%%%%%%%%%%%%%%%%%%%%%%%%%%%%%%%%%%%%%%%%%%%%%%%%%%%%%%%%%%%%%%%%%%%%%%%%%%%%%%%%%%%%%%%%%%%%%%%%%%%%%%%%%%%%%%%%%%
%%%%%%%%%%%%%%%%%%%%%%%%%%%%%%%%%%%%%%%%%%%%%%%%%%%%%PROOFS%%%%%%%%%%%%%%%%%%%%%%%%%%%%%%%%%%%%%%%%%%%%%%%%%%%%%%%%%%%%%%%%%%%%%%%%%%%%%%%%%%%%%%%%%%%%%%%
%%%%%%%%%%%%%%%%%%%%%%%%%%%%%%%%%%%%%%%%%%%%%%%%%%%%%%%%%%%%%%%%%%%%%%%%%%%%%%%%%%%%%%%%%%%%%%%%%%%%%%%%%%%%%%%%%%%%%%%%%%%%%%%%%%%%%%%%%%%%%%%%%%%%%%%%%%

\section{ \bf{Proof of Theorem \ref{thm:main}}}
\label{sec:proofs}
The goal of this Section is to prove Theorem \ref{thm:main}.

First, we extend the estimates of Proposition \ref{proposition:prop2} to a sufficiently wide class of test function by
using Helffer-Sj\"{o}strand functional calculus (\cite{HS}, \cite{D}) as in \cite{PRS}.
Let $f\in C^{l+1}(\R)$ decay at infinity sufficiently fast.  Then, one can write
\begin{equation}
 f(X_N)=-\frac{1}{\pi}\,\int_{\mathbb{C}}\frac{\partial \tilde{f}}{\partial \bar{z}}\, R_N(z)\,dxdy 
\quad,\quad\frac{\partial \tilde{f}}{\partial \bar{z}} := \frac{1}{2}\Big(\frac{\partial \tilde{f}}
{\partial x}+i\frac{\partial \tilde{f}}{\partial y}\Big)
\label{formula-H/S}
 \end{equation}
 where:
 \begin{itemize}
\item[i)]
 $z=x+iy$ with $x,y \in \mathbb{R}$;
 \item[ii)] $\tilde{f}(z)$ is the extension of the function $f$ defined as follows
  \begin{equation}\label{a.a. -extension}
  \tilde{f}(z):=\Big(\,\sum_{n=0}^{l}\frac{f^{(n)}(x)(iy)^n}{n!}\,\Big)\sigma(y);
\end{equation}
 here $\sigma \in C^{\infty}(\mathbb{R})$ is a nonnegative function equal to $1$ for $|y|\leq 1/2$ and equal to zero for $|y|\geq 1$.
 \end{itemize}
Using the definition of $\tilde{f}$ (see (\ref{a.a. -extension})) one can calculate 
\begin{eqnarray}
\frac{\partial \tilde{f}}{\partial \bar{z}}&=&\frac{1}{2}\Big(\frac{\partial \tilde{f}}{\partial x}+i\frac{\partial \tilde{f}}{\partial y}
\Big)\\
& =&\frac{1}{2} \Big(\,\sum_{n=0}^{l}\frac{f^{(n)}(x)(iy)^n}{n!}\,\Big)
i\frac{d\sigma}{dy}+\frac{1}{2}f^{(l+1)}(x)(iy)^l\frac{\sigma(y)}{l!}
\end{eqnarray}
and derive the crucial bound
\begin{equation}\label{estimate-derivative3}
\Big|\frac{\partial \tilde{f}}{\partial \bar{z}} (x+iy)\Big|\leq  Const \* \max\left(|\frac{d^jf}{dx^j}(x)|, \ 1\leq j \leq l+1\right) \*
|y|^l\quad.
\end{equation}

Directly following the calculations in Section 3 of \cite{PRS}, one obtains the following extention to a 
non-i.i.d. setting of Proposition 1.1 in \cite{PRS}.

\begin{proposition}
\label{proposition:prop3}
Let $X_N=\frac{1}{\sqrt{N}} W_N$ be a random real symmetric (Hermitian) Wigner matrix
(\ref{offdiagreal}), (\ref{diagreal}) (respectively (\ref{offdiagherm1}-\ref{diagherm}).
Then the following holds.

(i) Let $L$ be some positive number, $f\in C^7(\R)$ with compact support, and $supp(f) \subset [-L, +L].$
Then there exists a constant $Const(L, \sigma, \sigma_1, m_4)$ such that
\begin{align}
\label{Chto1}
& \big|\E (f(X_N)_{ii})-\int_{-2\sigma}^{2\sigma} f(x) \* \frac{1}{2 \pi \sigma^2} \sqrt{ 4 \sigma^2 - x^2}  \* dx\big|
\leq Const(L, \sigma, \sigma_1, m_4) \* \frac{\|f\|_{C^7([-L, L])}}{N}, \\
& 1\leq i \leq N. \nonumber
\end{align}

(ii) Let $f \in C^8(\R), $
then there exists a constant $Const(\sigma, \sigma_1, m_4)$ such that
\begin{align}
\label{Chto11}
& \big|\E (f(X_N)_{ii})-\int_{-2\sigma}^{2\sigma} f(x) \* \frac{1}{2 \pi \sigma^2} \sqrt{ 4 \sigma^2 - x^2}  \* dx \big|  \\
& \leq Const(\sigma, \sigma_1, m_4) \* \frac{\|f\|_{8,1,+}}{N} ,
\ 1\leq i \leq N. \nonumber
\end{align}
where $\|f\|_{n,1,+}$ is defined in (\ref{normasobolev}).

(iii) Let $f\in C^6(\R), $
then
\begin{equation}
\label{Chto12}
\big|\E (f(X_N)_{jk})\big|\leq Const(\sigma, \sigma_1, m_4) \* \frac{\|f\|_{6,1}}{N},
\ 1\leq j <k\leq N,
\end{equation}
where $\|f\|_{n,1}$ is defined in (\ref{normsob}).

(iv) Let $f \in C^4(\R),$ then
\begin{equation}
\label{Chto13}
\V (f(X_N)_{ij}) \leq Const(\sigma, \sigma_1, m_4) \* \frac{\|f\|^2_{4,1}}{N},  \ 1\leq i,j \leq N.
\end{equation}

(v) If 
\begin{equation*}
\sup_{i\not=j, N} \E |(W_N)_{ij}|^5 <\infty, \ \ \sup_{i, N} \E |(W_N)_{ii}|^3 <\infty,
\end{equation*}
and $f\in C^{10}(\R), $ then one can improve (\ref{Chto12}), namely
\begin{equation}
\label{Chto14}
|\E (f(X_N)_{jk})|\leq Const \* \frac{\|f\|_{10,1}}{N^{3/2}},
\ 1\leq j <k\leq N,
\end{equation}
where $Const$ depends on $\sup_{i\not=j, N} \E |(W_N)_{ij}|^5, $ and $ \sup_{i, N} \E |(W_N)_{ii}|^3.$
\end{proposition}

The next proposition is a corollary of Propositions \ref{prop:variance} and \ref{proposition:prop4}.
\begin{proposition}
\label{proposition:prop11}
Let $X_N=\frac{1}{\sqrt{N}} W_N$ be a random real symmetric Wigner matrix
(\ref{offdiagreal}), (\ref{diagreal}) such that
(\ref{LIND1}) is satisfied for some fixed $m\geq 1.$
Then there exists a random real symmetric Wigner matrix $T_N$ and a non-random positive sequence $\epsilon_N\to 0$ as $N \to \infty$ such that
the properties (\ref{mir0}-\ref{mir5}) from Lemma \ref{lemma:2}  are satisfied. 
In addition, for $s>3,$ there exists a constant $const_s$ that depends on $s, \sigma, \sigma_1,$ and $m_4$ such that
for $f\in \mathcal{H}_s$ 
\begin{equation}
\label{17}
\V [f (T_N/\sqrt{N})_{ij}] \leq const_s \frac{\|f\|^2_s}{N}, \ 1\leq i\leq m, \ 1\leq j\leq N.
\end{equation}
\end{proposition}

\begin{proof}
The existence of random real symmetric Wigner matrix $T_N$ satisfying (\ref{mir0}-\ref{mir5}) has been proven in Lemma \ref{lemma:2}.
Since $\P(X_N=T_N)\to 1$ as $N \to \infty, $ we can assume without loss of generality that $T_N=X_N.$

Let us first consider the diagonal case $i=j.$ Without loss of generality, one can assume $i=1.$
Define a random spectral measure  
\begin{equation*}
\mu(dx, \omega):= \sum_{l=1}^N \delta(x-\lambda_l) \* |\phi_l(1)|^2,
\end{equation*}
where $\lambda_l, \ 1\leq l \leq N, $ are the eigenvalues of $X_N$ and $\phi_l,  1\leq l \leq N, $ are the corresponding normalized eigenvectors.
Since by the result by Latala \cite{L}
\begin{equation*}
\sup_N \E \|X_N\| <\infty,
\end{equation*}
we have 
\begin{equation*}
\E \int |x| \* \mu(dx,\omega) =\E (|X_N|)_{11} <\infty,
\end{equation*}
one can apply Proposition \ref{prop:variance} and obtain
\begin{equation}
\label{vazhno10}
\V [f(X_N)_{11}] \leq Const_s \* \|f\|_s^2 \* \int_0^{\infty} dy\* e^{-y}\* y^{2s-1} \* \int_{-\infty}^{\infty} 
V [ (R_N(x+iy))_{11}] \* dx.
\end{equation}

To estimate the integral $ \int_{-\infty}^{\infty} V [ (R_N(x+iy))_{11}] \* dx$ in (\ref{vazhno10}), one uses the upper bound (\ref{odinnadtsat300})
in Proposition \ref{proposition:prop4} to obtain 
\begin{align}
\label{333}
& \frac{P_4(y^{-1})}{N}\* \E \*\int_{-\infty}^{+\infty} \|R_N(x+iy)\|^2 dx \\
\label{444}
& + \frac{P_4(y^{-1})}{N}\* \E \*\int_{-\infty}^{+\infty} \|R_N(x+iy)\|^{3/2} dx.
\end{align}
We will treat the first term (\ref{333}).  The second term (\ref{444}) can be treated in a similar fashion.
For $x\in [-\|X_N\|, +\|X_N\|], $ we use the trivial bound
\begin{equation*}
\|R_N(x+iy)\|^2\leq \frac{1}{y^2}.
\end{equation*}
For $|x|> \|X_N\|, $ we write 
\begin{equation*}
\|R_N(x+iy)\|^2 \leq \frac{1}{(x-\|X_N\|)^2 +y^2}.
\end{equation*}
Thus, 
\begin{equation}
\label{555}
\int_{-\infty}^{+\infty} \|R_N(x+iy)\|^2 dx \leq \frac{2\*\|X_N\|}{y^2} + \frac{\pi}{y}.
\end{equation}
Since (\cite{L})
\begin{equation*}
\sup_N \E \|X_N\|<\infty,
\end{equation*}
we obtain
\begin{equation}
\label{777}
\V [f(X_N)_{11}] \leq Const_s \* \frac{\|f\|_s^2}{N} \* \int_0^{\infty} dy\* e^{-y}\* y^{2s-1} \*P_4(y^{-1}) \* 
\left(\frac{const_1}{y^2}+\frac{const_2}{y^{1/2}} \right).
\end{equation}
If $s>3, $ the integral in (\ref{777}) converges.

In the off-diagonal case $i\not=j,$ one can consider the (complex-valued) measure
\begin{equation*}
\mu(dx, \omega):= \sum_{l=1}^N \delta(x-\lambda_l) \*\overline{\phi_l(i)} \*\phi_l(j),
\end{equation*}
write it as a linear combination of probability measures, and apply Proposition \ref{prop:variance} to each probability measure in the linear 
combination. Proposition \ref{proposition:prop11} is proven.
\end{proof}

Now, we are ready to prove Theorem \ref{thm:main}.
Let $m$ be a fixed positive integer. Denote by  $W_N^{(m)}$ the $m\times m$ upper-left corner submatrix of $W_N,$
and by $R^{(m)}_N(z)$ the $m\times m$ upper-left corner of the resolvent matrix $R_N(z).$
Our next step is to compute the limiting distribution of the normalized entries of $R^{(m)}_N(z)$  in the limit $N \to \infty.$
In the i.i.d. sertting, this was done in Theorem 1.1 (real symmetric case) 
and Theorem 1.5 (Hermitian case) in \cite{PRS}.  Below, we extend these results to the non-i.i.d. setting.  We start with the real symmetric case.
Define
\begin{align}
\label{ups}
& \Upsilon_N(z):=\sqrt{N} \* \left(R^{(m)}(z)-g_{\sigma}(z)\*I_m \right), \ z \in \C \setminus [-2\*\sigma, 2\*\sigma],\\
\label{psi}
& \Psi_N(z):= \Upsilon_N(z) - g_{\sigma}^2(z)\*W_N^{(m)}= \sqrt{N} \* \left(R^{(m)}(z)-g_{\sigma}(z)\*I_m \right)-g_{\sigma}^2(z)\*W_N^{(m)}.
\end{align}

Cleraly, $\Upsilon_N(z)$ and $\Psi_N(z)$  are random function with values in the space complex symmetric $m\times m$ matrices.
(real symmetric $m\times m$ matrices for real $ x$).
Define
\begin{equation}
\label{padova1}
\varphi(z,w):= \int_{-2\sigma}^{2\sigma} \frac{1}{z-x}\* \frac{1}{w-x}  \* \frac{1}{2 \pi \sigma^2} \sqrt{ 4 \sigma^2 - x^2}  \* dx=
\left \{\begin{array} {r@{\quad:\quad}l} 
-\frac{g_{\sigma}(w)-g_{\sigma}(z)}{w-z} & \text{if} \ \ w\not=z, \\
-g_{\sigma}'(z)& \text{if} \ \ w=z. \end{array} \right.
\end{equation}
for $z,w \in \C \setminus [-2\*\sigma, 2\*\sigma].\ $  
One can write $\varphi(z,w)=\E\left( \frac{1}{z-\eta}\*\frac{1}{w-\eta} \right), \ $
where $\eta$ is a Wigner semicircle law (\ref{polukrug}) random variable.
Let
\begin{align}
\label{padova2}
& \varphi_{++}(z,w):= \int_{-2\sigma}^{2\sigma} \Re \frac{1}{z-x}\* \Re\frac{1}{w-x}  
\* \frac{1}{2 \pi \sigma^2} \sqrt{ 4 \sigma^2 - x^2} \* dx\\
& =\frac{1}{4}\* \left(\varphi(z,w)+\varphi(\bar{z},\bar{w})+\varphi(\bar{z},w)+\varphi(z,\bar{w})\right), \nonumber\\
\label{padova3}
& \varphi_{--}(z,w):= \int_{-2\sigma}^{2\sigma} \Im \frac{1}{z-x}\* \Im\frac{1}{w-x}  
\* \frac{1}{2 \pi \sigma^2} \sqrt{ 4 \sigma^2 - x^2}  \* dx\\
& = -\frac{1}{4}\* \left(\varphi(z,w)+\varphi(\bar{z},\bar{w})-\varphi(\bar{z},w)-\varphi(z,\bar{w})\right), \nonumber\\
\label{padova4}
& \varphi_{+-}(z,w):= \int_{-2\sigma}^{2\sigma} \Re \frac{1}{z-x}\* \Im\frac{1}{w-x}  
\* \frac{1}{2 \pi \sigma^2} \sqrt{ 4 \sigma^2 - x^2}  \* dx\\
& =-\frac{i}{4}\* \left(\varphi(z,w)+\varphi(\bar{z},w)-\varphi(\bar{z},\bar{w})-\varphi(z,\bar{w})\right). \nonumber
\end{align}

\begin{theorem}
\label{thm:resreal}
Let $X_N=\frac{1}{\sqrt{N}} W_N$ be a random real symmetric Wigner matrix
(\ref{offdiagreal}), (\ref{diagreal}).  Let $m$ be a fixed positive integer and assume that for $1\leq i\leq m$ the conditions (\ref{LIND1}) 
and (\ref{AI}) are satisfied.  Also assume that
the Lindeberg type condition (\ref{lind1}) for the fourth moments  of the off-diagonal entries and 
the Lindeberg type condition (\ref{diagreal1}) for the second moments  of the diagonal entries
are satisfied.

Then the random field 
$ \Psi_N(z)$ in (\ref{psi}) converges in finite-dimensional distributions to the random field 
\begin{equation}
\label{functionalconv}
\Psi(z) = g_{\sigma}^2(z)\*Y(z),
\end{equation}
where
$Y(z)=\left(Y_{ij}(z)\right), Y_{ij}(z)=Y_{ji}(z), \ 1\leq i,j \leq m, $  is the Gaussian random field such that 

\begin{align}
\label{dispersii1}
& \Cov(\Re Y_{ii}(z), \Re Y_{ii}(w))= \kappa_4(i)\* \Re g_{\sigma}(z) \* \Re g_{\sigma}(w) +2\*\sigma^4\* \varphi_{++}(z,w), \\
\label{dispersii2}
& \Cov(\Im Y_{ii}(z), \Im Y_{ii}(w))= \kappa_4(i)\* \Im g_{\sigma}(z) \*\Im g_{\sigma}(w) +
2\*\sigma^4\* \varphi_{--}(z,w), \\
\label{dispersii3}
& \Cov(\Re Y_{ii}(z), \Im Y_{ii}(w) )= \kappa_4(i)\* \Re g_{\sigma}(z) \* \Im g_{\sigma}(w)  
+ 2 \*\sigma^4\* \varphi_{+-}(z,w), \\
\label{dispersii4}
& \Cov(\Re Y_{ij}(z),  \Re Y_{ij}(w))= \sigma^4\* \varphi_{++}(z,w), \ i\not=j,\\
\label{dispersii5}
& \Cov(\Im Y_{ij}(z), \Im Y_{ij}(w))= \sigma^4\* \varphi_{--}(z,w), \ i\not=j,\\
\label{dispersii6}
& \Cov(\Re Y_{ij}(z), \Im Y_{ij}(w) )= \sigma^4\* \varphi_{+-}(z,w), \ i\not=j,
\end{align}
where  $\kappa_4(i)=m_4(i)-3\sigma^4, \ 1\leq i\leq m,$ and $m_4(i)$ is defined in (\ref{AI}).

In addition, for any finite $r\geq 1, $ the
entries $Y_{i_lj_l}(z_l), \ 1\leq i_l\leq j_l \leq m, \ 1\leq l\leq r,$ are independent if for any $1\leq l_1\not=l_2\leq r$ one has 
$(i_{l_1}, j_{l_1})\not=(i_{l_2}, j_{l_2}).$
\end{theorem}

Now, we consider the Hermitian case.  As before, we define by (\ref{psi}) 
the matrix-valued random field $\Psi_N(z), \ z \in \C \setminus [-2\sigma, 2\sigma].\ \Psi_N(x)$ is Hermitian for real $ x$ and, more generally, 
$\Psi_N(z)=\Psi_N(\bar{z})^*. $

\begin{theorem}
\label{thm:resherm}
Let $X_N=\frac{1}{\sqrt{N}} W_N$ be a random real Hermitian Wigner matrix
(\ref{offdiagherm1}-\ref{diagherm}).  Let $m$ be a fixed positive integer and assume that for $1\leq i\leq m$ the conditions (\ref{LIND1}) 
and (\ref{AI}) are satisfied.  Also assume that 
the Lindeberg type condition (\ref{lind1}) for the fourth moments  of the off-diagonal entries and 
the Lindeberg type condition (\ref{diagreal1}) for the second moments  of the diagonal entries
are satisfied.

Then the random field 
$ \Psi_N(z) $ converges in finite-dimensional distributions to the random field 
\begin{equation}
\label{functionalconv1}
\Psi(z) = g_{\sigma}^2(z)\*Y(z),
\end{equation}
where 
$Y(z)=\left(Y_{ij}(z)\right), \ 1\leq i,j \leq m, \ $  is the Gaussian random field such that 

\begin{align}
\label{dispersii11}
& \Cov(\Re Y_{ii}(z), \Re Y_{ii}(w))= \kappa_4(i)\* \Re g_{\sigma}(z) \* \Re g_{\sigma}(w) +\sigma^4\* \varphi_{++}(z,w), \\
\label{dispersii12}
& \Cov(\Im Y_{ii}(z), \Im Y_{ii}(w))= \kappa_4(i)\* \Im g_{\sigma}(z) \*\Im g_{\sigma}(w) +
\sigma^4\* \varphi_{--}(z,w), \\
\label{dispersii13}
& \Cov(\Re Y_{ii}(z), \Im Y_{ii}(w) )= \kappa_4(i)\* \Re g_{\sigma}(z) \* \Im g_{\sigma}(w)  
+ \sigma^4\* \varphi_{+-}(z,w), \\
\label{dispersii14}
& \Cov(\Re Y_{ij}(z),  \Re Y_{ij}(w))= \frac{1}{2}\*\sigma^4\* (\varphi_{++}(z,w)+\varphi_{--}(z,w)), \ i\not=j, \\
\label{dispersii15}
& \Cov(\Im Y_{ij}(z), \Im Y_{ij}(w))= \frac{1}{2}\*\sigma^4\* (\varphi_{++}(z,w)+\varphi_{--}(z,w)), \ i\not=j,\\
\label{dispersii16}
& \Cov(\Re Y_{ij}(z), \Im Y_{ij}(w) )= \frac{1}{2}\*\sigma^4\* (\varphi_{+-}(z,w)-\varphi_{+-}(w,z))       , \ i\not=j.
\end{align}
where $\kappa_4(i)= m_4(i)- 2\sigma^4, \ 1\leq i\leq m,$ and $m_4(i)$ is defined in (\ref{AI}).

In addition, for any finite $r\geq 1,$ the 
entries $Y_{i_lj_l}(z_l), \ 1\leq i_l\leq j_l \leq m, \ 1\leq l\leq r,$ are independent provided 
$(i_{l_1}, j_{l_1})\not=(i_{l_2}, j_{l_2}) $ for $1\leq l_1\not=l_2\leq r.$
\end{theorem}
\begin{remark}
If the distribution of the entries of $W_N$ does not depend on $N,$  
the random field 
\begin{equation*}
\Upsilon_N(z)=\sqrt{N} \* \left(R^{(m)}(z)-g_{\sigma}(z)\*I_m \right), \ z \in \C \setminus [-2\sigma, 2\sigma]
\end{equation*}
converges in finite-dimensional distributions to
$ g_{\sigma}^2(z)\*(Y(z) +W^{(m)}),$ 
where $Y(z)$ is independent from $W^{(m)}.$
\end{remark}
Below, we sketch the proof of Theorem \ref{thm:resreal}.  The proof in the Hermitian case is very similar.
\begin{proof}
As in \cite{PRS}, one can write
\begin{equation}
\label{ugol}
R_N^{(m)}(z) = \left (z\*I_m - X^{(m)} - M^*\* \tilde{R}\* M \right)^{-1}=\left (z\*I_m - \frac{1}{\sqrt{N}}\*W_N^{(m)} - 
M^*\* \tilde{R}\* M \right)^{-1} ,
\end{equation}
where $X_N^{(m)}$ is the $m\times m$ upper-left corner submatrix of $X_N, \ \tilde{X}^{(N-m)}$ is the \\
$(N-m)\times(N-m) $ 
lower-right corner submatrix of $X_N, $
\begin{equation*}
\tilde{R}_N(z)= \left(z\* I_{N-m} - \tilde{X}^{(N-m)}\right)^{-1},
\end{equation*} 
is the resolvent of $\tilde{X}^{(N-m)}, $ and $M$ is the the $(N-m) \times m \ $ lower-left corner submatrix of $X_N.$
We will denote by $x^{(1)}, \ldots, x^{(m)} \in \R^{N-m} $ the (column) vectors that form $M,$ and by $M^*$
the adjoint matrix of $M.$

It follows from Proposition \ref{proposition:prop1} that $\tilde{R}_N(z)$ is well defined for any fixed $z \in \C \setminus [-2\sigma, 2\sigma]$ 
with probability going to $1.$ 

Define the $m\times m$ matrix $\Gamma_N(z)$ as
\begin{equation}
\label{gammamatrix}
(\Gamma_N)_{ij}(z)=
(W_N)_{ij} + \sqrt{N} \* \left(\langle x^{(i)}, \tilde{R}_N(z) \* x^{(j)} \rangle - \sigma^2\*g_{\sigma}(z)\*\delta_{ij} \right),
\ 1\leq i,j \leq m.
\end{equation}
Then
\begin{equation}
\label{chetyreugla}
\Gamma_N(z)=W_N^{(m)} + Y_N(z),
\end{equation}
where 
\begin{equation}
\label{mnogouglov}
(Y_N(z))_{ij}= Y_{ij}(z)= \sqrt{N} \* \left(\langle x^{(i)}, \tilde{R}(z) \* x^{(j)} \rangle - 
\sigma^2\*g_{\sigma}(z)\*\delta_{ij} \right),\ 1\leq i,j \leq m.
\end{equation}

Equations (\ref{ugol}) and (\ref{gammamatrix}) imply
\begin{equation}
\label{qatar}
R^{(m)}(z)= \left(\frac{1}{g_{\sigma}(z)}\* I_m -\frac{1}{\sqrt{N}} \* \Gamma_N(z)\right)^{-1}.
\end{equation}
It will follow from the Central Limit Theorem for Quadratic Forms (see discussion below and the Appendix) that  $\|\Gamma_N(z)\|$  
is bounded in probability. This would imply that 
\begin{equation}
\label{uefa}
\Upsilon_N(z)=\sqrt{N} \* \left(R^{(m)}(z)- g_{\sigma}(z) \*I_m\right)= g^2_{\sigma}(z)\* \Gamma_N(z) +o(1),
\end{equation}
in probability (meaning that the error term goes to zero in probability), and 
\begin{equation}
\label{uefa1}
\Psi_N(z)= \sqrt{N} \* \left(R^{(m)}(z)-g_{\sigma}(z)\*I_m \right)-g_{\sigma}^2(z)\*W_N^{(m)}= g^2_{\sigma}(z)\*Y_N(z) +o(1),
\end{equation}
in probability.

To estimate $\|\Gamma_N(z)\|,$ where $\Gamma_N(z)=W_N^{(m)} + Y_N(z),$ we note that for fixed $m, \ \|W_N^{(m)}\|$ is bounded in probability.  
Let us consider in more detail $Y_N(z).$   Assume that
$z$ is fixed and $\Im z \not=0.$
It follows from 
\begin{equation*}
\E Y_N(z)=  \sqrt{N} \*\sigma^2\*(g_N(z)-g_{\sigma}(z))\*I_m,
\end{equation*}
and Proposition \ref{proposition:prop2} that $\E Y_N(z)\to 0.$
Thus,
\begin{equation}
\label{jap}
Y_N(z)_{ij}=  \sqrt{N} \* \left(\langle x^{(i)}, \tilde{R}(z) \* x^{(j)} \rangle - \E \langle x^{(i)}, \tilde{R}(z) \* x^{(j)} \rangle \right)+o(1),
\ 1\leq i,j\leq m.
\end{equation}
We note that the vectors $x^{(i)}, \ 1\leq i \leq m, $ are independent from $\tilde{R}(z).$  
In the Appendix, we point out that
the Central Limit Theorem for Quadratic Forms also holds in the non-i.i.d. case under the conditions on the entries of
$x^{(i)}, \ 1\leq i\leq m,$  that are equivalent to (\ref{LIND1}).  This implies that $\|Y_N (z)\|$ is bounded in probability, and therefore
$\|\Gamma_N(z)\|$ is bounded in probability as well, which implies (\ref{uefa}-\ref{uefa1}).

To study the finite-dimensional distributions of $Y_N(z),$ we fix a positive integer $p\geq 1,$ and consider $z_1, \ldots, z_p \in \C \setminus \R.$
Taking into account (\ref{jap}), the problem is reduced to the question about the joint distribution of the entries 
$\sqrt{N} \big[(R_N(z_l))_{i_l, j_l}- \E (R_N(z_l))_{i_l, j_l}\big], \ 1\leq i_l\leq j_l\leq m, \ 1\leq l\leq p.$  
To this end, we apply Theorem \ref{thm:mult-real-qf} in the Appendix with $r=m,$ and
\begin{equation}
\label{summamatric}
B_N^{s,t}= \sum_{l=1}^p \left( a_{s,t}^{(l)}\* \Re( \tilde{R}(z_l)) + b_{s,t}^{(l)} \* \Im( \tilde{R}(z_l))\right), \ \ 1\leq s \leq t \leq m,
\end{equation}
where $ a_{s,t}^{(l)}, \ b_{s,t}^{(l)},  \ 1\leq s \leq t \leq m, \ 1\leq l \leq p, \ $  are arbitrary real numbers,
and 
\begin{equation*}
y_N^{(s)}= \frac{\sqrt{N}}{\sigma}\*x^{(s)}, \ 1\leq s\leq m.
\end{equation*}

The condition (i) of Theorem \ref{thm:mult-real-qf} is equivalent to (\ref{LIND1}).
The condition (ii) is automatically satisfied as long as $\Im z_l\not=0, \ 1\leq l\leq m.$
Conditions (iii) and (iv) are equivalent to

\begin{align}
\label{91}
& \frac{1}{N-m}\* \Tr \left(\Re( \tilde{R}(z)) \* \Re( \tilde{R}(w)) \right) \to  \varphi_{++}(z,w),\\
\label{92}
& \frac{1}{N-m}\*\Tr\left(\Im( \tilde{R}(z)) \* \Im( \tilde{R}(w)) \right) \to  \varphi_{--}(z,w),\\
\label{93}
& \frac{1}{N-m}\* \Tr \left(\Re( \tilde{R}(z) \*\Im( \tilde{R}(w)\right)  \to  \varphi_{+-}(z,w),  \\
\label{94}
& \frac{1}{N-m} \sum_{j=m+1}^N \kappa_4((W_N)_{ij})\* (\Re( \tilde{R}(z)))_{jj} \* (\Re( \tilde{R}(w)))_{jj} \to  \kappa_4(i) \*
\Re(g_{\sigma}(z)) \*\Re(g_{\sigma}(w)),\\
\label{95}
& \frac{1}{N-m} \sum_{j=m+1}^N \kappa_4((W_N)_{ij})\*(\Im( \tilde{R}(z)))_{jj} \* (\Im( \tilde{R}(z)))_{jj} \to 
\kappa_4(i)\*\Im(g_{\sigma}(z)) \* \Im(g_{\sigma}(w)),\\
\label{96}
& \frac{1}{N-m} \sum_{j=m+1}^N \kappa_4((W_N)_{ij})\* (\Re( \tilde{R}(z)))_{jj}\*(\Im( \tilde{R}(w)))_{jj} \to \kappa_4(i)\*
\Re(g_{\sigma}(z))\*\Im(g_{\sigma}(w)),
\end{align}
for $z,w \in \C \setminus [-2\sigma, 2\sigma], \ 1\leq i\leq m,$ 
where $\varphi_{++}(z,w), \varphi_{--}(z,w),$ and $\varphi_{+-}(z,w)$ are defined in (\ref{padova2}-\ref{padova4}), and
the convergence is in probability.  To make the formulas (\ref{94}-\ref{96}) look less cumbersome, we label the diagonal entries
of the $(N-m)\times(N-m)$ matrices $\Re(\tilde{R}(z)), \ \Im(\tilde{R}(z))$ by index $j=m+1, \ldots, N.$

The conditions (\ref{91}-\ref{93}) follow from the semicircle law,
and (\ref{94}-\ref{96}) follow from the estimates (\ref{odinnadtsat100}) and (\ref{odinnadtsat102}) in Proposition \ref{proposition:prop2}.
The details are left to the reader.  
Theorem \ref{thm:mult-real-qf} now implies that $Y_N(z)$ converges in finite-dimensional distributions to $Y(z)$ for $\Im z \not=0.$ 
For $z\in \R\setminus [-2\sigma, 2\sigma], $ one can replace $\tilde{R}(z)$ by $h(X_N)\*\tilde{R}(z),$ where $h$ satisfies (\ref{amerika}) and 
repeat the arguments above since $\P(\tilde{R}(z)\not=h(X_N)\*\tilde{R}(z))\to 0$ as $N \to \infty.$
\end{proof}
To complete the proof of Theorem \ref{thm:main}, we first restrict our attention to the four time continuously differentiable test functions 
with compact support.
Let $f \in C^4_c(\R).$  It follows from Theorem \ref{thm:real} and Proposition \ref{proposition:prop3} 
that the result of Theorem \ref{thm:main} holds for finite linear combinations
\begin{equation}
\label{1200}
\sum_{l=1}^k a_l\* h_l(x)\*\frac{1}{z_l-x}, \ z_l \not\in [-2\sigma, 2\sigma], \ 1\leq l \leq k, 
\end{equation}
where $h_l\in C^{\infty}_c(\R), \ 1\leq l\leq k,$ satisfies (\ref{amerika}).
By Stone-Weierstrass theorem (see e.g \cite{RS}), one can approximate an arbitrary $C^4_c(\R)$ by functions of the form (\ref{1200}).  Moreover, 
if $supp(f)\subset [-A, A], $ one can choose the approximating sequance in such a way that $ supp(h_l) \subset [-A-1, A+1].$ 
Applying the bound (\ref{Chto13}) in Proposition \ref{proposition:prop3}, we show that
\begin{equation*}
\V [ \sqrt{N} \* ( f(X_N)_{ij}- \sum_{l=1}^k a_l\* (h_l(X_N) R_N)_{ij} ) ]
\end{equation*}
can be made arbitrary small uniformly in $N,$ which finishes the proof for $f \in C^4_c(\R).$

To extend the proofs to the case of $ f \in \mathcal{H}_s,$  for some $s>3,$ we use the estimate (\ref{17}) in
Proposition \ref{proposition:prop11} and approximate such $f$ by a sequence $\{f_n\}_{n\geq 1}$ such that
\begin{equation}
\|f-f_n\|_s\to 0, \ \text{as} \ n\to \infty, \ f_n\in C^4_c(\R), \ n\geq 1.
\end{equation}
This finishes the proof of Theorem \ref{thm:main}.

\begin{appendix}
\section{ \bf{Central Limit Theorem for Quadratic Forms}}
\label{sec:clt}
The appendix is devoted to the formulation of the CLT type results for the quadratic forms $y_N^\ast B y_N$ where 
$y_N$ is a random $N$-vector that contains independent entries with finite 
fourth moment and $B$ is a random $N \times N$ Hermitian matrix.  The formulated results and their proofs are similar to the results in 
\cite{BY}, \cite{CDF} (see the appendix by Baik and Silverstein), and \cite{BGM} since the arguments presented there work with small changes in 
the non-i.i.d. setting as well.

First we present the case where the entries of $Y_N$ are complex and then the case where the entries are real.  

\begin{theorem}[Central Limit Theorem for Quadratic Forms] \label{thm:qf}
Let $B=(b_{ij})_{1 \leq i,j \leq N}$ be a $N \times N$ random Hermitian matrix and $y_N=(y_{Nj})_{1 \leq j \leq N}$ be an independent vector of size 
$N$ which contains independent complex standardized entries such that
$ \sup_{N,j} \E|y_{Nj}|^4 =m_4 < \infty $
and $\E(y_{Nj}^2) = 0$.  Assume that
\begin{enumerate}[(i)]
\item for all $\epsilon>0$, 
\begin{equation} \label{Y-lf}
        \frac{1}{N}\sum_{j=1}^N \E\left[ \left| |y_{Nj}|^2-1 \right|^2 \indicator{|y_{Nj}|^2-1| > \epsilon \sqrt{N}}\right] \longrightarrow 0
\end{equation}
as $N \rightarrow \infty$,
\item there exists a constant $a>0$ (not depending on $N$) such that $\|B\| \leq a$,
\item $\frac{1}{N} \Tr B^2$ converges in probability to a number $a_2$,
\item $\frac{1}{N} \sum_{i=1}^N b_{ii}^2 \* \kappa_4(y_{Ni})$ converges in probability to a number $a_1,$  
\end{enumerate}
where
\begin{equation}
\label{kappaiherm}
\kappa_4(y_{Ni}):=\E|y_{Ni}|^4 -2, \ 1\leq i \leq N.
\end{equation}
Then the random variable $\frac{1}{\sqrt{N}}(y_N^\ast B y_N - \Tr B)$ converges in distribution to a Gaussian random variable with mean zero and variance
\begin{equation*}
       v^2= a_1 + a_2.
\end{equation*}
\end{theorem}

\begin{theorem}[Central Limit Theorem for Real Quadractic Forms] \label{thm:real-qf}
Let $B=(b_{ij})_{1 \leq i,j \leq N}$ be a $N \times N$ random real symmetric matrix and $y_N=(y_{Nj})_{1 \leq j \leq N}$ be an independent vector of 
size $N$ which contains independent real standardized entries with $\sup_{N,j}\E|y_{Nj}|^4 =m_4 < \infty$.  Assume that conditions 
(i)-(iv) hold as in Theorem 
\ref{thm:qf} with
\begin{equation}
\label{kappaireal}
\kappa_4(y_{Ni}):=\E|y_{Ni}|^4 -3, \ 1\leq i \leq N.
\end{equation}

Then the random variable $\frac{1}{\sqrt{N}}(y_N^\ast B y_N - \Tr B)$ converges in distribution to a Gaussian random variable with mean 
zero and variance
\begin{equation*}
        v^2=a_1 + 2a_2.
\end{equation*}
\end{theorem}

Finally, we formulate the multidimensional versions of Theorems \ref{thm:qf} and \ref{thm:real-qf}.  
We again consider the real and complex cases separately.  

\begin{theorem} \label{thm:mult-qf}
Let $\{B^{s,t}: 1 \leq s,t \leq r\}$ be a family of $N \times N$ random matrices with the property that $(B^{s,t})^\ast = B^{t,s}$.  
Let $\{y_N^{(s)} : 1 \leq s \leq r\}$ be a family of independent $N$-vectors with independent complex standardized entries where $y^{(s)}_N = 
(y_{Nj}^{(s)})_{1 \leq j \leq N}$, $\sup_{N,j}\*\E|y_{Nj}^{(s)}|^4 = m_4 < \infty$, and  $\E[(y_{Nj}^{(s)})^2] = 0$.  Further assume that
\begin{enumerate}[(i)]
\item for all $\epsilon>0$, 
\begin{equation*} 
        \frac{1}{N}\sum_{j=1}^N \E\left[ \left| |y_{Nj}^{(s)}|^2-1 \right|^2 \indicator{||y_{Nj}^{(s)}|^2-1| > \epsilon \sqrt{N}}\right] \longrightarrow 0
\end{equation*}
as $N \rightarrow \infty$ for each $1 \leq s \leq r$,
\item there exists a constant $a>0$ (not depending on $N$) such that $\max_{1 \leq s,t \leq r}\|B^{s,t}\| \leq a$,
\item $\frac{1}{N} \Tr ((B^{s,t})^\ast B^{s,t})$ converges in probability to a number $a_2(s,t)$,
\item $\frac{1}{N} \sum_{i=1}^N (B^{s,s})_{ii}^2 \* \kappa_4(y_{Ni}^{(s)})$ converges in probability to a number $a_1(s),$  
\end{enumerate}
where 
\begin{equation*}
\kappa_4(y_{Ni}^{(s)})= \E |y_{Ni}^{(s)}|^4-2, \ 1\leq i\leq N.
\end{equation*}
Then the $r \times r$ matrix
\begin{equation*}
        G_N = \frac{1}{\sqrt{N}} \left( (y_N^{(s)})^\ast B^{s,t} y_N^{(t)} - \delta_{s,t} \Tr B^{s,t} \right)_{1 \leq s,t \leq r}
\end{equation*}
converges in distribution to an $r \times r$ Hermitian matrix $G$ such that the linearly independent entries are statistically independent and 
$\Re(G_{st}), \Im(G_{st}) \sim \mathcal{N}\left(0, \frac{1}{2}a_2(s,t) \right)$ for $s \neq t$ and $G_{ss} \sim 
\mathcal{N}(0, a_1(s) + a_2(s,s))$.
\end{theorem}

\begin{theorem} \label{thm:mult-real-qf}
Let $\{B^{s,t}: 1 \leq s,t \leq r\}$ be a family of $N \times N$ real random matrices with the property that $(B^{s,t})\T = B^{t,s}$.  
Let $\{y_N^{(s)} : 1 \leq s \leq r\}$ be a family of independent $N$-vectors with independent real standardized entries where 
$y^{(s)}_N = (y_{Nj}^{(s)})_{1 \leq j \leq N}$ and $\sup_{N,j} \* \E|y_{Nj}^{(s)}|^4 = m_4 < \infty.$  Further assume that the conditions (i)-(iv) from 
Theorem 
\ref{thm:mult-qf} hold with
\begin{equation*}
\kappa_4(y_{Ni}^{(s)})= \E |y_{Ni}^{(s)}|^4-3, \ 1\leq i\leq N.
\end{equation*}

Then the $r \times r$ matrix
\begin{equation*}
        G_N = \frac{1}{\sqrt{N}} \left( (y_N^{(s)})^\ast B^{s,t} y_N^{(t)} - \delta_{s,t} \Tr B^{s,t} \right)_{1 \leq s,t \leq r}
\end{equation*}
converges in distribution to an $r \times r$ symmetric matrix $G$ such that the linearly independent entries are statistically independent and 
$G_{s,t} \sim \mathcal{N}(0, a_2(s,t))$ for $s \neq t$ and $G_{s,s} \sim \mathcal{N}\left(0, a_1(s)+2a_2(s,s) \right)$.  
\end{theorem}

\end{appendix}

\end{document}